\documentclass{article}
\usepackage[cp1251]{inputenc}
\usepackage[russian]{babel}
\usepackage{amssymb,amsfonts,amsmath,amsthm,amscd}
\usepackage{graphicx}
\usepackage{enumerate}
\usepackage{soul}
\usepackage[matrix,arrow,curve]{xy}
\usepackage{longtable}
\usepackage{array}
\usepackage{bm}
\RequirePackage{russlh}
\RequirePackage{mathlh}

\newdimen\symskip
\newdimen\defskip
\newdimen\parind
\parind=\parindent
\newdimen\leftmarge
\newdimen\theoremshape
\topsep\defskip
\makeatletter
\newcommand*{\clei}{\nobreak\hskip\z@skip}

\newcommand{\name}[1]{`#1'}
\renewcommand{\:}{\textup{:}}
\renewcommand{\~}{\textup{;}}

\DeclareRobustCommand*{\д}{\clei\hbox{-}\clei}
\newcommand{\no}{}
\renewcommand{\@listI}{\settowidth\labelwidth{\labheadi{\no}}\listipar{\parind}{\labelwidth}}
\newcommand{\listivpar}{\topsep\defskip\partopsep0pt\parsep-\parskip\itemsep0.5\topsep}
\newcommand{\listipar}[2]{\rightmargin0pt\leftmargin#1\labelsep#1\advance\labelsep-#2\itemindent0pt\listivpar}
\renewcommand{\@listii}{\settowidth\labelwidth{\labheadii{\@roman{\no}}}\listiipar{\parind}{\labelwidth}}
\newcommand{\listiivpar}{\topsep0.5\defskip\partopsep0pt\parsep-\parskip\itemsep0.5\topsep}
\newcommand{\listiipar}[2]{\rightmargin0pt\leftmargin#1\labelsep#1\advance\labelsep-#2\itemindent0pt\listiivpar}
\def\thempfn{\ifcase\value{footnote}1\or *\or **\or ***\else\@ctrerr\fi}
\renewcommand\footnoterule{%
  \kern-3\p@
  \hrule\@width1in
  \kern2.6\p@}
\makeatother

\makeatletter

\renewcommand{\@biblabel}[1]{[#1]}
\renewenvironment{thebibliography}[1]
     {\renewcommand{\refname}{References}%
      \renewcommand{\No}{}%
      \section*{\refname}%
      \@mkboth{\MakeUppercase\refname}{\MakeUppercase\refname}%
      \list{\@biblabel{\@arabic\c@enumiv}}%
           {\itemsep\baselineskip
            \leftmargin\parind
            \settowidth\labelwidth{\@biblabel{#1}}%
            \labelsep\parind\advance\labelsep-\labelwidth
            \@openbib@code
            \usecounter{enumiv}%
            \let\p@enumiv\@empty
            \renewcommand\theenumiv{\@arabic\c@enumiv}}%
      \sloppy
      \clubpenalty4000
      \@clubpenalty\clubpenalty
      \widowpenalty4000%
      \sfcode`\.\@m}
     {\def\@noitemerr
       {\@latex@warning{Empty `thebibliography' environment}}%
      \endlist}

\def\@maketitle{%
  \newpage
  \vskip0.5em%
  UDK \udk%
  \vskip0.5em%
  MSC \msc%
  \vskip1em%
  \begin{center}\bf%
  \let\footnote\thanks%
   {\Large\@author\par}%
   \vskip1.5em%
   {\LARGE\@title\par}%
   \vskip1em%
   {\large\@date}%
  \end{center}%
  \par
  \vskip1.5em}

\def\@title{\@latex@warning@no@line{No \noexpand\title given}}

\makeatother

\makeatletter

\renewcommand\sectionmark[1]{%
 \markright{%
  \ifnum \c@secnumdepth >\z@
   \thesection. \ %
  \fi
 #1}}%

\renewcommand{\section}{\@startsection{section}{1}{0pt}%
{5.5ex plus .5ex minus .2ex}{1.5ex plus .3ex}%
{\center\normalfont\Large\bfseries\sffamily\bom}}
\renewcommand{\subsection}{\@startsection{subsection}{2}{0pt}%
{4.5ex plus .4ex minus .2ex}{0.75ex plus .2ex}%
{\center\normalfont\large\bfseries\sffamily\bom}}
\renewcommand{\subsubsection}{\@startsection{subsubsection}{3}{0pt}%
{2.5ex plus .5ex minus .2ex}{1ex plus .2ex}%
{\center\normalfont\bfseries\sffamily\bom}}

\newcommand{\Ss}{\textup{\S\,}}

\def\@postskip@{\hskip.5em\relax}

\def\postsection{.\@postskip@}

\def\postsubsection{.\@postskip@}

\def\postsubsubsection{.\@postskip@}

\def\postparagraph{.\@postskip@}

\def\postsubparagraph{.\@postskip@}

\def\@seccntformat#1{\csname pre#1\endcsname\csname the#1\endcsname\csname post#1\endcsname}

\makeatother

\renewcommand{\thesection}{\textup{\arabic{section}}}

\newcommand{\parr}{\par\addvspace{\defskip}}
\newcommand{\theo}[2]{\newtheorem{#1}{#2}[section]}
\newcommand{\deff}[2]{\newenvironment{#1}{\parr\textbf{#2.}}{\parr}}
\theo{cas}{Case}
\theo{problem}{Problem}
\theo{theorem}{Theorem}
\theo{lemma}{Lemma}
\theo{prop}{Proposition}
\theo{stm}{Statement}
\theo{fact}{Fact}
\theo{imp}{Corollary}
\theo{ex}{Example}
\deff{df}{Definition}
\deff{note}{Note}
\deff{denote}{Notation}
\deff{denotes}{Notations}
\deff{hint}{Hint}
\deff{answer}{Answer\:}

\makeatletter
\def\@begintheorem#1#2[#3]{%
  \deferred@thm@head{\the\thm@headfont \thm@indent
    \@ifempty{#1}{\let\thmname\@gobble}{\let\thmname\@iden}%
    \@ifempty{#2}{\let\thmnumber\@gobble}{\let\thmnumber\@iden}%
    \@ifempty{#3}{\let\thmnote\@gobble}{\let\thmnote\@iden}%
    \thm@notefont{\bfseries\upshape}%
    \indent%
    \thm@swap\swappedhead\thmhead{#1}{#2}{#3}%
    \the\thm@headpunct
    \thmheadnl 
    \hskip\thm@headsep
  }%
  \ignorespaces}
\renewenvironment{proof}{\setcounter{cas}{0}\parr\pushQED{\qed}\normalfont$\square\quad$}{\setcounter{cas}{0}\popQED\@endpefalse\parr}

\makeatother

\newcommand{\labheadi}[1]{\textup{#1)}}
\newcommand{\labheadii}[1]{\textup{(#1)}}
\newcommand{\labhi}[1]{\labheadi{\arabic{#1}}}

\newenvironment{nums}[1]{\renewcommand{\no}{#1}\begin{enumerate}}{\end{enumerate}}
\newcommand{\eqn}[1]{\begin{equation}#1\end{equation}}
\newcommand{\equ}[1]{\begin{equation*}#1\end{equation*}}

\newcommand{\case}[1]{\begin{cases}#1\end{cases}}
\newcommand{\cask}[1]{\begin{casks}#1\end{casks}}

\newcommand{\rbmat}[1]{\begin{pmatrix}#1\end{pmatrix}}

\makeatletter

\def\LT@makecaption#1#2#3{%
  \LT@mcol\LT@cols c{\hbox to\z@{\hss\parbox[t]\LTcapwidth{%
    \sbox\@tempboxa{#1{#2. }#3}%
    \ifdim\wd\@tempboxa>\hsize
      #1{#2. }#3%
    \else
      \hbox to\hsize{\hfil\box\@tempboxa\hfil}%
    \fi
    \endgraf\vskip\baselineskip}%
  \hss}}}

\@addtoreset{equation}{section}
\@addtoreset{footnote}{section}

\newenvironment{casks}{%
  \matrix@check\casks\env@casks
}{%
  \endarray\right.%
}
\def\env@casks{%
  \let\@ifnextchar\new@ifnextchar
  \left\lbrack
  \def\arraystretch{1.2}%
  \array{@{}l@{\quad}l@{}}%
}

\newcounter{numt}
\newcounter{col}
\newcounter{coll}

\makeatother

\renewcommand{\ge}{\geqslant}
\renewcommand{\le}{\leqslant}
\newcommand{\fa}{\,\forall\,}

\newcommand{\es}{\varnothing}
\newcommand{\eqi}{\equiv}
\newcommand{\neqi}{\not\eqi}
\newcommand{\subs}{\subset}
\newcommand{\sups}{\supset}

\newcommand{\sm}{\setminus}
\newcommand{\wo}{\backslash}
\newcommand{\cln}{\colon}

\newcommand{\chk}{\check}

\newcommand{\Ra}{\Rightarrow}
\newcommand{\Lra}{\Leftrightarrow}

\newcommand{\os}[1]{\overset{#1}}
\newcommand{\Inn}[1]{\smash{\os{\circ}{\smash{#1}\vph{^{_{^{_{c}}}}}}}\vph{#1}}
\newcommand{\wt}{\widetilde}

\newcommand{\Cn}[2]{\rbmat{#1\\#2}}

\newcommand{\mb}[1]{$\bm{#1}$}

\newcommand{\sums}[1]{\sum\limits_{{#1}}}

\newcommand*{\bw}[1]{#1\nobreak\discretionary{}{\hbox{$\mathsurround=0pt #1$}}{}}
\newcommand{\sco}{,\ldots,}
\newcommand{\spl}{\bw+\ldots\bw+}

\newcommand{\seq}{\bw=\ldots\bw=}

\newcommand{\sge}{\bw\ge\ldots\bw\ge}

\newcommand{\ha}[1]{\left\langle#1\right\rangle}
\newcommand{\ba}[1]{\bigl\langle#1\bigr\rangle}

\newcommand{\br}[1]{\bigl(#1\bigr)}
\newcommand{\Br}[1]{\Bigl(#1\Bigr)}
\newcommand{\bbr}[1]{\biggl(#1\biggr)}
\newcommand{\ter}[1]{\textup{(}#1\textup{)}}
\newcommand{\bgm}[1]{\bigl|#1\bigr|}
\newcommand{\Bm}[1]{\Bigl|#1\Bigr|}
\newcommand{\hn}[1]{\left\|#1\right\|}
\newcommand{\bn}[1]{\bigl\|#1\bigr\|}
\newcommand{\Bn}[1]{\Bigl\|#1\Bigr\|}

\newcommand{\bc}[1]{\bigl\{#1\bigr\}}
\newcommand{\BC}[1]{\Bigl\{#1\Bigr\}}

\newcommand{\bbl}{\bigm\wo}

\newcommand{\mbb}{\mathbb}
\newcommand{\mbf}{\mathbf}
\newcommand{\mcl}{\mathcal}
\newcommand{\mfr}{\mathfrak}

\newcommand{\R}{\mbb{R}}

\newcommand{\Z}{\mbb{Z}}
\newcommand{\N}{\mbb{N}}

\newcommand{\Cbb}{\mbb{C}}
\newcommand{\Hbb}{\mbb{H}}

\newcommand{\Eb}{\mbb{E}}

\newcommand{\Pc}{\mcl{P}}

\newcommand{\ggt}{\mfr{g}}

\newcommand{\tgt}{\mfr{t}}
\newcommand{\pd}{\partial}

\newcommand{\al}{\alpha}
\newcommand{\be}{\beta}
\newcommand{\ga}{\gamma}
\newcommand{\Ga}{\Gamma}
\newcommand{\de}{\delta}
\newcommand{\De}{\Delta}
\newcommand{\ep}{\varepsilon}
\newcommand{\la}{\lambda}
\newcommand{\La}{\Lambda}

\newcommand{\ph}{\varphi}
\newcommand{\om}{\omega}
\newcommand{\Om}{\Omega}

\DeclareMathOperator{\Lie}{Lie}

\DeclareMathOperator{\ad}{ad}

\DeclareMathOperator{\rk}{rk}

\DeclareMathOperator{\conv}{conv}

\newcommand{\GL}{\mbf{GL}}

\newcommand{\Or}{\mbf{O}}
\newcommand{\SO}{\mbf{SO}}
\newcommand{\Sp}{\mbf{Sp}}
\newcommand{\Spin}{\mbf{Spin}}

\newcommand{\glg}{\mfr{gl}}

\newcommand{\sog}{\mfr{so}}
\newcommand{\un}{\mfr{u}}
\newcommand{\sug}{\mfr{su}}
\newcommand{\spg}{\mfr{sp}}

\newcommand{\bom}{\boldmath}

\newcommand{\hph}[1]{\hphantom{#1}}
\newcommand{\vph}[1]{\vphantom{#1}}

\newcommand{\thra}{\twoheadrightarrow}

\begin{document}

\author{O.\,G.\?Styrt}
\title{On the orbit spaces\\
of irreducible representations\\
of simple compact Lie groups\\
of types \mb{B}, \mb{C}, and~\mb{D}}
\date{}
\newcommand{\udk}{512.815.1+512.815.6+512.816.1+512.816.2}
\newcommand{\msc}{22E46+17B10+17B20+17B45}

\maketitle

{\leftskip\parind\rightskip\parind
It is proved that the orbit space of an irreducible representation of a~simple connected compact Lie group of type $B$, $C$, or~$D$ can be a~smooth
manifold only in two cases.

\smallskip

\textbf{Key words\:} Lie group, topological quotient of an action.\par}

\section{Introduction}\label{introd}

The paper is a~continuation of \cite{My1,My2,My3}. First of all, we give two basic definitions that also played a~key role in those papers.

\begin{df} A~continuous map of smooth manifolds is called \textit{piecewise smooth} if it takes any smooth submanifold onto a~finite union of smooth
submanifolds.
\end{df}

In particular, any proper smooth map of smooth manifolds is piecewise smooth.

Consider a~differentiable action of a~compact Lie group~$G$ on a~smooth manifold~$M$.

\begin{df} The quotient $M/G$ is \textit{a~smooth manifold} if the topological quotient $M/G$ admits a~structure of a~smooth manifold such that
the factorization map $M\bw\to M/G$ is piecewise smooth.
\end{df}

Now we can formulate the main problem.

Let $V$ be a~real vector space and $G\subs\GL(V)$ a~compact linear group. Like in \cite{My1,My2,My3}, the question is whether the quotient $V/G$ is
a~topological manifold and, also, whether it is a~smooth manifold. Following \cite{My1,My2,My3}, we will say \name{manifold} instead of \name{topological
manifold}.

Denote by~$V_{\Cbb}$ the complex space $V\otimes\Cbb$, by~$\ggt$ the linear Lie algebra $\Lie G\subs\glg(V)$, and by~$\ggt_{\Cbb}$ the complex linear Lie
algebra $\ggt\otimes\Cbb\subs\glg(V_{\Cbb})$.

The problem has almost been solved in the following cases\:
\begin{nums}{-1}
\item $[\ggt,\ggt]=0$ (see~\cite{My1})\~
\item $\ggt\cong\sug_2$ (see~\cite{My2})\~
\item $\ggt\cong\sug_{r+1}$, $r=\rk\ggt>1$, and the linear Lie algebra $\ggt\subs\glg(V)$ is irreducible (see~\cite{My3}).
\end{nums}
In this paper we consider the case when $\ggt$ is a~simple compact Lie algebra of one of types $B_r$ ($r>1$), $C_r$ ($r>2$), and $D_r$ ($r>3$), and the
linear Lie algebra $\ggt\subs\glg(V)$ is irreducible.

Denote by~$R'$ the representation of a~complex reductive Lie group dual to its representation~$R$.

When meaning indecomposable systems of simple roots, we will use the numeration of the simple roots given in \cite{Bo}, \cite[Table\,1]{VO}, and
\cite[Table\,1]{Elsh}, denoting by~$\ph_i$ the fundamental weight with the number~$i$.

The space~$V$ has a~$G$\д invariant scalar product. Hence, we will suppose that $V$ is a~Euclidean space and the group~$G$ acts on it by orthogonal
operators. Thus, $G\subs\Or(V)$.

We obviously have $r=\rk\ggt>1$.

Denote by~$\wt{R}$ the tautological representation $\ggt_{\Cbb}\cln V_{\Cbb}$. One of the following cases occurs\:
\begin{nums}{-1}
\item $\wt{R}=R$, where $R$ is an irreducible representation\~
\item $\wt{R}=R+R'$, where $R$ is an irreducible representation.
\end{nums}
In the latter case, the space~$V$ has a~$\ggt$\д invariant complex structure inducing a~complex representation $\ggt_{\Cbb}\cln V$ isomorphic to the
representation~$R$.

In this paper we will prove Theorems \ref{B_main}---\ref{D_main1}.

\begin{theorem}\label{B_main} If $G\cong B_r$, then the condition that $V/G$ be a~smooth manifold excludes all but the following cases\:
\begin{nums}{-1}
\item the representation~$\wt{R}$ of the algebra~$\ggt_{\Cbb}$ coincides with the representation~$R$ and is isomorphic to one of the representations
$\ad$, $R_{\ph_1}$, and~$R_{2\ph_1}$\~
\item $\wt{R}=R+R'$, $r\in\{2,5\}$, and the representation~$R$ is isomorphic to the representation~$R_{\ph_r}$.
\end{nums}
\end{theorem}

\begin{theorem}\label{B_main1} If $G$ is a~connected Lie group of type~$B_r$ and the representation~$R$ is isomorphic to one of the representations
$\ad$, $R_{\ph_1}$, $R_{2\ph_1}$, and $R_{\ph_2}$ \ter{$r=2$}, then $V/G$ is not a~manifold.
\end{theorem}

\begin{imp}\label{B_main2} If $G$ is a~connected Lie group of type~$B_r$ and $V/G$ is a~smooth manifold, then $r=5$, and the representation~$R$ is
isomorphic to the representation~$R_{\ph_5}$.
\end{imp}

\begin{theorem}\label{C_main} If $r>2$ and $G\cong C_r$, then the condition that $V/G$ be a~smooth manifold excludes all but the following cases\:
\begin{nums}{-1}
\item the representation~$\wt{R}$ of the algebra~$\ggt_{\Cbb}$ coincides with the representation~$R$ and is isomorphic to one of the representations
$\ad$, $R_{\ph_2}$, and $R_{\ph_4}$ \ter{$r=4$}\~
\item $\wt{R}=R+R'$, and the representation~$R$ is isomorphic to the representation~$R_{\ph_1}$.
\end{nums}
\end{theorem}

\begin{theorem}\label{C_main1} If $G$ is a~connected Lie group of type~$C_r$, $r>2$, and the representation~$R$ is isomorphic to one of the
representations $\ad$, $R_{\ph_2}$, $R_{\ph_1}$, and $R_{\ph_4}$ \ter{$r=4$}, then $V/G$ is not a~manifold.
\end{theorem}

\begin{imp}\label{C_main2} If $G$ is a~connected Lie group of type~$C_r$, $r>2$, then $V/G$ is not a~smooth manifold.
\end{imp}

\begin{theorem}\label{D_main} If $r>3$ and $G\cong D_r$, then the condition that $V/G$ be a~smooth manifold excludes all but the following cases\:
\begin{nums}{-1}
\item the representation~$\wt{R}$ of the algebra~$\ggt_{\Cbb}$ coincides with the representation~$R$ and is isomorphic \ter{up to an outer automorphism}
to one of the representations $\ad$, $R_{\ph_1}$, $R_{2\ph_1}$, and $R_{\ph_8}$ \ter{$r=8$}\~
\item $\wt{R}=R+R'$, $r\in\{5,6\}$, and the representation~$R$ of the algebra~$\ggt_{\Cbb}$ is isomorphic \ter{up to an outer automorphism} to the
representation~$R_{\ph_r}$.
\end{nums}
\end{theorem}

\begin{theorem}\label{D_main1} If $G$ is a~connected Lie group of type~$D_r$, $r>3$, and the representation~$R$ of the algebra~$\ggt_{\Cbb}$ is
isomorphic \ter{up to an outer automorphism} to one of the representations $\ad$, $R_{\ph_1}$, $R_{2\ph_1}$, and $R_{\ph_r}$ \ter{$r=5,8$}, then $V/G$ is
not a~manifold.
\end{theorem}

\begin{imp}\label{D_main2} If $G$ is a~connected Lie group of type~$D_r$, $r>3$, and $V/G$ is a~smooth manifold, then $r=6$, and the representation~$R$
of the algebra~$\ggt_{\Cbb}$ is isomorphic \ter{up to an outer automorphism} to the representation~$R_{\ph_6}$.
\end{imp}

In the cases mentioned in Corollaries \ref{B_main2} and~\ref{D_main2} (i.\,e. $G=G^0$ is of type~$B_5$ and $R\cong R_{\ph_5}$, or
$G=G^0$ is of type~$D_6$ and $R\cong R_{\ph_6}$), the problem has not been resolved.

The further text is structured in the following way. In \Ss\ref{promain}, we prove Theorems \ref{B_main}, \ref{C_main}, and~\ref{D_main}. The proofs use
lower estimates on ranks of some operators of the linear algebra $\ggt\subs\glg(V)$. Therefore, we should prove a number of quantitative relations for
the weight set of the representation~$R$ of the algebra~$\ggt_{\Cbb}$ that we do in \Ss\ref{facts}. Since this weight set is explicitly expressed through
the Weyl group action on the weight lattice, \Ss\ref{facts} does not deal with Lie algebras and involves just a~rigorous technical work with abstract
root systems of types $B$, $C$, and~$D$ in Euclidean spaces and the structure of orbits of the above-mentioned actions. So, the statements of
\Ss\ref{promain} essentially refer on those of \Ss\ref{facts}, but not vice-versa. As for Theorems \ref{B_main1}, \ref{C_main1}, and~\ref{D_main1}, they
are much easier proved in~\Ss\ref{promain1}.

\section{Proofs of the main results}\label{promain}

This section is devoted to proving Theorems \ref{B_main}, \ref{C_main}, and~\ref{D_main}.

First of all, we should give special and recall well-known notations and facts.

For any non-zero vector~$\al$ in a~Euclidean space~$\Eb$, denote by~$\chk{\al}$ the vector $\frac{2\al}{(\al,\al)}\in\Eb$.

In~\cite{My0}, for each indecomposable system of simple roots~$\Pi$, we defined some subset $\pd\Pi\subs\Pi$. All indecomposable systems of simple
roots~$\Pi$ such that $\pd\Pi\ne\Pi$ are listed in \cite[\Ss4,~Table\,1]{My0}, including the explicit statement of the subsets $\pd\Pi\subs\Pi$. In
particular,
\eqn{\label{bou}\begin{aligned}
(\Pi\cong B_r)&&&\Ra&&&&
\pd\Pi=\case{
\{\al_1\},&r\ge5;\\
\{\al_1,\al_r\},&r<5;}\\
(\Pi\cong C_r)&&&\Ra&&&&
\pd\Pi=\{\al_1,\al_2\};\\
(\Pi\cong D_r)&&&\Ra&&&&
\pd\Pi=\case{
\{\al_1\},&r\ge7;\\
\{\al_1,\al_{r-1},\al_r\},&r<7.}
\end{aligned}}

For a~linear representation of a~Lie group~$G$ (resp. a~Lie algebra~$\ggt$) in a~space~$V$, we denote the stabilizer (resp. the stationary subalgebra) of
a~vector $v\in V$ by~$G_v$ (resp. by~$\ggt_v$).

\begin{df} A~representation of a~compact Lie group~$G$ in a~real space~$V$ is called \textit{polar} if there exists a~subspace $V'\subs V$ such that
$GV'=V$ and $V'\cap\br{T_v(Gv)}=0$ ($v\in V'$).
\end{df}

\begin{lemma}\label{pol} The orbit space of any faithful polar representation of a~nontrivial connected compact Lie group is homeomorphic to a~closed
half-space.
\end{lemma}

\begin{proof} See Lemma~2.10 in~\cite[\Ss2]{My3}.
\end{proof}

Return to the conventions of \Ss\ref{introd}.

\begin{lemma}\label{didi} Assume that $V/G$ is a~smooth manifold. For each vector $v\in V$ such that $\rk\ggt_v=1$ and for each non-zero vector
$\xi\in\ggt_v$, we have $\dim(\xi V)\le\dim[\xi,\ggt]+6$.
\end{lemma}

\begin{proof} See Lemma~2.4 in~\cite[\Ss2]{My3}.
\end{proof}

Set $\de:=1\in\R$ if the representation~$R$ of the algebra~$\ggt_{\Cbb}$ is orthogonal and $\de:=2\in\R$ otherwise.

There exist a~complex space~$\wt{V}$ and an irreducible representation $\ggt_{\Cbb}\cln\wt{V}$ such that the representation $\ggt\cln V$ is either the
realification or a~real form of the representation $\ggt\cln\wt{V}$. We clearly have $(\de=1)\Lra(\wt{V}=V\otimes\Cbb)$ and $(\de=2)\Lra(\wt{V}=V)$.

Fix a~maximal commutative subalgebra~$\tgt$ of the algebra~$\ggt$ and the Cartan subalgebra $\tgt_{\Cbb}:=\tgt\otimes\Cbb$ of the algebra~$\ggt_{\Cbb}$.
Thus we get a~root system $\De\subs\tgt_{\Cbb}^*$ and its Weyl group $W\subs\GL(\tgt_{\Cbb}^*)$. Clearly,
$\ha{\De}=\bc{\la\in\tgt_{\Cbb}^*\cln\la(\tgt)\subs i\R}\subs\tgt_{\Cbb}^*=\ha{\De}\oplus i\ha{\De}$.

For any root $\al\in\De$, denote by~$h_{\al}$ the vector of the space~$\tgt_{\Cbb}$ identified with the vector
$\chk{\al}\bw\in\ha{\De}\bw\subs\tgt_{\Cbb}^*$ via the Cartan scalar product.

Fix a~positive root system $\De^+\subs\De$ and the corresponding simple system $\Pi\subs\De$ and Weyl chamber $C\subs\ha{\De}\subs\tgt_{\Cbb}^*$.

Denote by $P$ and~$Q$ the lattices $\bc{\la\in\ha{\De}\cln(\la,\chk{\al})\in\Z\,\fa\al\in\De}\subs\ha{\De}$ and $\ha{\De}_{\Z}\bw\subs\ha{\De}$
respectively. We have $Q\subs P$. Let $\la\in(P\cap C)\sm\{0\}$ be the highest weight of the representation~$R$ of the algebra~$\ggt_{\Cbb}$ with respect
to the simple system $\Pi\subs\De\subs\tgt_{\Cbb}^*$.

Set $\La:=W\la\subs P$ and $\Inn{\La}:=\conv(\La)\cap(\La+Q)\subs P$. It is easy to see that the set of weights of the representation~$R$ of the
algebra~$\ggt_{\Cbb}$ equals the subset $\Inn{\La}\subs P$.

We will use the notation~$\hn{\cdot}$ for the order of an arbitrary subset of the weight set $\Inn{\La}\subs P$ of the representation~$R$ \textit{taking
multiplicities of weights into account}.

\begin{lemma}\label{bms} Assume that $V/G$ is a~smooth manifold. If $v\in V$, $\rk\ggt_v=1$, and $\xi\bw\in(\ggt_v\cap\tgt)\sm\{0\}$, then
$\de\cdot\Bn{\bc{\la'\in\Inn{\La}\cln\la'(\xi)\ne0}}\le\Bm{\bc{\al\in\De\cln\al(\xi)\ne0}}+6$.
\end{lemma}

\begin{proof} See Lemma~2.8 in~\cite[\Ss2]{My3}.
\end{proof}

\begin{lemma}\label{nosm} Let $\Om\subs\La$ be some subset and $H$ the subspace $\ha{\Om}_{\Cbb}\subs\tgt_{\Cbb}^*$. If
\begin{gather*}
\dim_{\Cbb}(\tgt_{\Cbb}^*/H)=1;\quad\quad\de\cdot\bn{\Inn{\La}\sm H}>|\De\sm H|+6;\quad\quad(\Om-\Om)\cap\De=\es;\\
(\de=1)\quad\Ra\quad\Br{\br{2\la\notin\De\cup(\De+\De)}\land\br{(\Om+\Om)\cap\De=\es}},
\end{gather*}
then $V/G$ is not a~smooth manifold.
\end{lemma}

\begin{proof} See Corollary~2.4 in~\cite[\Ss2]{My3}.
\end{proof}

Set $\Pi_{\la}:=\bc{\al\in\Pi\cln(\la,\chk{\al})\ne0}\subs\Pi$. Denote by~$\Pc$ the family of all indecomposable simple systems
$\Pi'\subs\Pi\subs\tgt_{\Cbb}^*$ of order $r-2$.

\begin{fact}\label{cov} If $r>2$, then the set~$\Pi$ equals the union of all its subsets $\Pi'\in\Pc$.
\end{fact}

\begin{proof} See Statement~3.1 in~\cite[\Ss3]{My0}.
\end{proof}

\subsection{Proof of Theorem~\ref{B_main}}

In this subsection, we will prove Theorem~\ref{B_main}.

Suppose that $\ggt\cong\sog_n$, where $n:=2r+1\in\N$.

Clearly, $\De\subs\tgt_{\Cbb}^*$ is an indecomposable root system of type~$B_r$.

We will prove Theorem~\ref{B_main} \name{by contradiction}\: assume that $V/G$ is a~smooth manifold and the linear algebra $R(\ggt_{\Cbb})$ is not
isomorphic to any of the linear algebras $\ad(\ggt_{\Cbb})$, $R_{\ph_1}(B_r)$, $R_{2\ph_1}(B_r)$ ($r>1$), and $R_{\ph_r}(B_r)$ ($r=2,5$). Since
$\ph_1=\ep_1$ and the highest root $\ep_1+\ep_2$ are the only roots contained in~$C$, and $\ph_2\in\De$ for $r>2$, the latter means that
\eqn{\label{B_sug}
\begin{array}{c}
\la\notin\De\cup\{\ph_2,2\ph_1\};\\
(r=5)\quad\Ra\quad(\la\ne\ph_5).
\end{array}}

\begin{lemma}\label{B_pila} Let $\Pi'\in\Pc$ be a~system of simple roots such that $(r>2)\Ra(\Pi_{\la}\cap\Pi'\ne\es)$ and let
$\Pi_0\subs\De\subs\ha{\De}\subs\tgt_{\Cbb}^*$ be the simple system corresponding to the positive root system
$\De^+\cap\ba{\{\la\}\cup\Pi'}\subs\ha{\De}\subs\tgt_{\Cbb}^*$. Suppose that
\eqn{\label{B_con}
\begin{array}{lll}
\br{r\in\{3,4,6\}}\quad&\Ra\quad&(\la\ne\ph_r);\\
(r=3)\quad&\Ra\quad&(\la\ne\ph_1+\ph_3).
\end{array}}
Then
\begin{nums}{-1}
\item $r>2$, $\Pi_{\la}\cap\Pi'=\{\al\}\subs\Pi$, where $\al\in\pd\Pi'\subs\Pi$, and $(\la,\chk{\al})=1$\~
\item if the simple system $\Pi_0\subs\ha{\De}$ is indecomposable, then
\eqn{\label{alb}
\begin{array}{cc}
\Pi_{\la}\cap\Pi'=\{\al\}\subs\Pi;\quad&\quad\al\in\Pi\cap\pd\Pi_0;\\
\fa\be\in\Pi_0\sm\{\al\}\quad&\quad(\la,\chk{\be})=0.
\end{array}}
\end{nums}
\end{lemma}

\begin{proof} Clearly, $H:=\ba{\{\la\}\cup\Pi'}\subs\ha{\De}\subs\tgt_{\Cbb}^*$ is an $(r-1)$\д dimensional (real) subspace, and, hence, the
intersection of the kernels of all its linear functions has the form $\Cbb\xi\subs\tgt_{\Cbb}$, $\xi\in\tgt\sm\{0\}$.

Assume that the claim does not hold.

By Lemma~3.4 in~\cite[\Ss3]{My0}, there exists a~vector $v\in V$ such that $\xi\in\ggt_v$ and $\rk\ggt_v=1$.

Lemmas \ref{didi} and~\ref{bms} imply
\eqn{\label{B_les}
\begin{array}{c}
\de\cdot\rk_{\Cbb}\br{R(\xi)}=\dim(\xi V)\le\dim[\xi,\ggt]+6=\dim_{\Cbb}[\xi,\ggt_{\Cbb}]+6;\\
\bgm{\Inn{\La}\sm H}\le\de\cdot\bgm{\Inn{\La}\sm H}\le|\De\sm H|+6.
\end{array}}
Using Lemma~\ref{B_lah} and Equations \eqref{B_sug}, \eqref{B_con}, and~\eqref{B_les}, we obtain that $r=3$ and $\la=2\ph_3$ or $r=2$ and
$\la\in\ph_2+\br{Q\sm\{0\}}$.

Suppose that $r=3$ and $\la=2\ph_3$.

We have $\Pi_{\la}\cap\Pi'\ne\es$ and $\Pi_{\la}=\{\al_3\}\subs\Pi$. Hence, $\Pi'=\{\al_3\}$. Also, $(\la,\chk{\al}_1)\bw=2\cdot(\ph_3,\chk{\al}_1)\bw=0$
and $(\al_3,\chk{\al}_1)=0$, implying $h_{\al_1}\in\Cbb\xi$, and, by~\eqref{B_les},
$\rk_{\Cbb}\br{R(h_{\al_1})}\bw\le\dim_{\Cbb}[h_{\al_1},\ggt_{\Cbb}]+6$. The adjoint representation (resp. the representation~$R$) of the algebra
$\ggt_{\Cbb}=\sog_n(\Cbb)$, $n=7$, is isomorphic to the external square (resp. the external cube) of its tautological representation, and the eigenvalues
of the semisimple operator $h_{\al_1}\in\sog_n(\Cbb)$ are exactly the numbers $0$, $1$, and~$-1$ with multiplicities $n_0:=3$, $n_+:=2$, and $n_-:=2$
respectively. Therefore,
\equ{\begin{aligned}
\dim_{\Cbb}[h_{\al_1},\ggt_{\Cbb}]&=\Cn{n}{2}-\bbr{\Cn{n_0}{2}+\Cn{n_+}{1}\cdot\Cn{n_-}{1}}=14;\\
\rk_{\Cbb}\br{R(h_{\al_1})}&=\Cn{n}{3}-\bbr{\Cn{n_0}{3}+\Cn{n_0}{1}\cdot\Cn{n_+}{1}\cdot\Cn{n_-}{1}}=22,
\end{aligned}}
contradicting $\rk_{\Cbb}\br{R(h_{\al_1})}\le\dim_{\Cbb}[h_{\al_1},\ggt_{\Cbb}]+6$.

Now suppose that $r=2$ and $\la\in\ph_2+\br{Q\sm\{0\}}$.

We have $(\la,\chk{\al}_2)\in(\ph_2,\chk{\al}_2)+(\al_1,\chk{\al}_2)\cdot\Z+(\al_2,\chk{\al}_2)\cdot\Z=1+2\Z$, and thus the representation~$R$ of the
algebra~$\ggt_{\Cbb}$ is symplectic. Hence, $\de=2$. By~\eqref{B_les}, $2\cdot\bgm{\Inn{\La}\sm H}\bw\le|\De\sm H|+6$, contradicting Lemma~\ref{B_lah2}.

This completes the proof.
\end{proof}

\begin{imp}\label{B_pil} Assume that \eqref{B_con} holds. Then $r>2$. Also,
\begin{nums}{-1}
\item for any $\al\in\Pi_{\la}$ we have $(\la,\chk{\al})=1$\~
\item for any simple system $\Pi'\in\Pc$ we have $\Pi_{\la}\cap\br{\Pi'\sm(\pd\Pi')}=\es$\~
\item on the Dynkin diagram of the simple system~$\Pi$ each two distinct roots of the subset $\Pi_{\la}\subs\Pi$ correspond to vertices, the path between
which contains at least $r-2$ edges.
\end{nums}
\end{imp}

\begin{lemma}\label{B_Bt} Suppose that $(\la,\chk{\al}_r)=0$. Then $(\la,\chk{\al}_1)\seq(\la,\chk{\al}_{r-2})=0$.
\end{lemma}

\begin{proof} By assumption, \eqref{B_con} holds.

Suppose that $(\la,\chk{\al}_i)\ne0$ for some $i\in\{1\sco r-2\}$ (in particular, $r>2$).

The simple system $\Pi':=\{\al_1\sco\al_{r-2}\}\subs\Pi$ lies in the family~$\Pc$. Besides, $\Pi_{\la}\bw\cap\Pi'\bw\ne\es$,
$H:=\ba{\{\la\}\cup\Pi'}=\ha{\al_r}^{\perp}$, $\al:=\al_{r-1}+\al_r\in\De^+\cap H$, $\chk{\al}=2\ep_{r-1}$, $\chk{\al}_r=2\ep_r$,
$\chk{\al}_{r-1}=\ep_{r-1}-\ep_r$, $\chk{\al}=2\chk{\al}_{r-1}+\chk{\al}_r$, $(\la,\chk{\al})\bw=2\cdot(\la,\chk{\al}_{r-1})+(\la,\chk{\al}_r)$. The
positive root system $\De^+\cap H\bw\subs\ha{\De}\bw\subs\tgt_{\Cbb}^*$ corresponds to the indecomposable simple system~$\Pi_0$ of type~$B_{r-1}$
containing the simple roots $\al_1\sco\al_{r-2},\al$ in the standard order. According to~\eqref{bou}, $\pd\Pi_0\subs\{\al_1,\al\}$. By
Lemma~\ref{B_pila}, $i=1$, $(\la,\chk{\al}_1)=1$, $(\la,\chk{\al}_2)\seq(\la,\chk{\al}_{r-2})=0$, and, also, $(\la,\chk{\al})=0$,
$2\cdot(\la,\chk{\al}_{r-1})+(\la,\chk{\al}_r)\bw=(\la,\chk{\al})\bw=0$, yielding $(\la,\chk{\al}_{r-1})=(\la,\chk{\al}_r)=0$. Hence, $\la=\ph_1\in\De$,
contradicting~\eqref{B_sug}.
\end{proof}

Applying Corollary~\ref{B_pil}, Lemma~\ref{B_Bt}, and Equations \eqref{bou} and~\eqref{B_sug}, we obtain that
\equ{
\cask{
r\in\{3,4,5,6\},&\la\in\{\ph_1+\ph_r,\ph_2+\ph_r\};\\
r=3,&\la=\ph_1+\ph_2+\ph_3;\\
r\in\{3,4,6\},&\la=\ph_r;\\
r=4,&\la=\ph_3.}}

Indeed, if \eqref{B_con} does not hold, then the claim is evident. Suppose that \eqref{B_con} holds. By Corollary~\ref{B_pil}, $r>2$ and
$(\la,\chk{\al})=1$ for any $\al\in\Pi_{\la}$.

Assume that $r\ge7$. Applying Equation~\eqref{bou} and Corollary~\ref{B_pil} for $\Pi':=\{\al_3\sco\al_r\}$, we obtain $(\la,\chk{\al}_k)=0$ for all
$k>3$. By Lemma~\ref{B_Bt}, we get $(\la,\chk{\al}_1)\bw=(\la,\chk{\al}_2)\bw=(\la,\chk{\al}_3)\bw=0$, so $\la=0$, contradiction. Hence,
$r\in\{3,4,5,6\}$.

Suppose that $(\la,\chk{\al}_r)=0$. By Lemma~\ref{B_Bt}, $\la=\ph_{r-1}$. It follows from~\eqref{B_sug} that $r>3$. If we had that $r\in\{5,6\}$, then on
taking $\Pi':=\{\al_3\sco\al_r\}$ again, we would arrive at $\al_{r-1}\notin\pd\Pi'$ by Equation~\eqref{bou}, in contradiction to Corollary~\ref{B_pil}.
Thus, $r=4$ in this case.

To conclude, we may assume that $(\la,\chk{\al}_r)\ne0$. By Corollary~\ref{B_pil}, $\{\al_r\}\bw\subs\Pi_{\la}\bw\subs\{\al_1,\al_2,\al_r\}$ and
$\br{\Pi_{\la}=\{\al_1,\al_2,\al_r\}}\Ra(r=3)$. Therefore, $\la\bw\in\{\ph_r,\ph_1\bw+\ph_r,\ph_2\bw+\ph_r\}$, or $r=3$ and $\la=\ph_1+\ph_2+\ph_3$. Note
also that, by~\eqref{B_sug}, if $\la=\ph_r$, then $r\ne5$ and thus $r\in\{3,4,6\}$.

\begin{cas} $r=3,4,5,6$ and $\la\in\{\ph_1+\ph_r,\ph_2+\ph_r\}$ or $r=3$ and $\la=\ph_1+\ph_2+\ph_3$.
\end{cas}

Using the inequality $\de\ge1$, Proposition~\ref{B_all}, and Lemma~\ref{nosm}, we come to a~contradiction to the suggestion that $V/G$ is a~smooth
manifold.

\begin{cas} $r=3,4$ and $\la=\ph_r$.
\end{cas}

The representation~$R_{\ph_r}$ of the complex simple Lie group $\Spin_n(\Cbb)$ is orthogonal and (see~\cite[\Ss3]{CD}) polar. Further, by
Lemma~\ref{B_phr}, $-E\in G^0\subs\Or(V)$, where $E$ denotes the identity operator, and thus $G=G^0$. According to Lemma~\ref{pol}, the quotient $V/G$ is
homeomorphic to a~closed half-space. At the same time, $V/G$ is a~smooth manifold, which is a~contradiction.

\begin{cas} $r=6$ and $\la=\ph_6$.
\end{cas}

The representation~$R$ of the algebra~$\ggt_{\Cbb}$ is symplectic. Hence, $\de=2$. Using Proposition~\ref{B_last} and Lemma~\ref{nosm}, we get
a~contradiction to the assumption that $V/G$ is a~smooth manifold.

\begin{cas} $r=4$ and $\la=\ph_3$.
\end{cas}

The representation~$R$ of the algebra~$\ggt_{\Cbb}$ is orthogonal implying $\de=1$. By Proposition~\ref{B_thi} and Lemma~\ref{nosm}, $V/G$ is not
a~smooth manifold, which contradicts the suggestion.

Thus, we have completely proved Theorem~\ref{B_main} (\name{by contradiction}).

\subsection{Proof of Theorem~\ref{C_main}}

In this subsection, we will prove Theorem~\ref{C_main}.

Suppose that $r>2$ and $\ggt\cong\un_r(\Hbb)$. Set $n:=2r\in\N$.

Clearly, $\De\subs\tgt_{\Cbb}^*$ is an indecomposable root system of type~$C_r$.

We will prove Theorem~\ref{C_main} \name{by contradiction}\: assume that $V/G$ is a~smooth manifold and the linear algebra $R(\ggt_{\Cbb})$ is not
isomorphic to any of the linear algebras $\ad(\ggt_{\Cbb})$, $R_{\ph_2}(C_r)$, $R_{\ph_1}(C_r)$ ($r>2$), and $R_{\ph_4}(C_4)$. The latter means that
\eqn{\label{C_sug}
\begin{array}{c}
\la\notin\De\cup\{\ph_1\};\\
(r=4)\quad\Ra\quad(\la\ne\ph_4).
\end{array}}

It is easy to see that $\de=1\in\R$ if $\la\in Q$ and $\de=2\in\R$ if $\la\notin Q$, where $\de$ is defined as in the beginning of \Ss\ref{promain}.

\begin{lemma}\label{C_pila} Let $\Pi'\in\Pc$ be a~system of simple roots such that $\Pi_{\la}\cap\Pi'\ne\es$ and let
$\Pi_0\bw\subs\De\bw\subs\ha{\De}\bw\subs\tgt_{\Cbb}^*$ be the simple system corresponding to the positive root system
$\De^+\bw\cap\ba{\{\la\}\bw\cup\Pi'}\bw\subs\ha{\De}\bw\subs\tgt_{\Cbb}^*$. Suppose that
\eqn{\label{C_con}
\br{r\in\{3,4\}}\quad\Ra\quad(\la\ne\ph_3).}
Then
\begin{nums}{-1}
\item $\Pi_{\la}\cap\Pi'=\{\al\}\subs\Pi$, where $\al\in\pd\Pi'\subs\Pi$, and $(\la,\chk{\al})=1$\~
\item if the simple system $\Pi_0\subs\ha{\De}$ is indecomposable, then \eqref{alb} holds.
\end{nums}
\end{lemma}

\begin{proof} Clearly, $H:=\ba{\{\la\}\cup\Pi'}\subs\ha{\De}\subs\tgt_{\Cbb}^*$ is an $(r-1)$\д dimensional (real) subspace, and, hence, the intersection
of the kernels of all its linear functions has form $\Cbb\xi\subs\tgt_{\Cbb}$, $\xi\in\tgt\sm\{0\}$.

Assume that the claim does not hold.

By Lemma~3.4 in~\cite[\Ss3]{My0}, there exists a~vector $v\in V$ such that $\xi\in\ggt_v$ and $\rk\ggt_v=1$.

It follows from Lemma~\ref{bms} that $\de\cdot\bgm{\Inn{\La}\sm H}\le|\De\sm H|+6$. Applying Lemma~\ref{C_lah} and Equation~\eqref{C_sug}, we obtain that
$r\in\{3,4\}$ and $\la=\ph_3$, contradicting~\eqref{C_con}.
\end{proof}

\begin{imp}\label{C_pil} Assume that \eqref{C_con} holds. Then
\begin{nums}{-1}
\item for any $\al\in\Pi_{\la}$ we have $(\la,\chk{\al})=1$\~
\item for any simple system $\Pi'\in\Pc$ we have $\Pi_{\la}\cap\br{\Pi'\sm(\pd\Pi')}=\es$\~
\item on the Dynkin diagram of the simple system~$\Pi$ each two distinct roots of the subset $\Pi_{\la}\subs\Pi$ correspond to vertices, the path
between which contains at least $r-2$ edges.
\end{nums}
\end{imp}

\begin{proof} Follows from Fact~\ref{cov} and Lemma~\ref{C_pila}.
\end{proof}

\begin{lemma}\label{C_Ct} If \eqref{C_con} holds and, also, $(\la,\chk{\al}_r)=0$, then $(\la,\chk{\al}_1)\seq(\la,\chk{\al}_{r-2})=0$.
\end{lemma}

\begin{proof} Assume that $(\la,\chk{\al}_i)\ne0$ for some $i\in\{1\sco r-2\}$.

The simple system $\Pi':=\{\al_1\sco\al_{r-2}\}\subs\Pi$ lies in the family~$\Pc$. Besides, $\Pi_{\la}\bw\cap\Pi'\bw\ne\es$,
$H:=\ba{\{\la\}\cup\Pi'}=\ha{\al_r}^{\perp}$, $\al:=2\al_{r-1}+\al_r\in\De^+\cap H$, $\chk{\al}=\ep_{r-1}$, $\chk{\al}_r=\ep_r$,
$\chk{\al}_{r-1}=\ep_{r-1}-\ep_r$, $\chk{\al}=\chk{\al}_{r-1}+\chk{\al}_r$, $(\la,\chk{\al})\bw=(\la,\chk{\al}_{r-1})+(\la,\chk{\al}_r)$. The positive
root system $\De^+\cap H\subs\ha{\De}\subs\tgt_{\Cbb}^*$ corresponds to the indecomposable simple system~$\Pi_0$ of type~$C_{r-1}$ containing the
simple roots $\al_1\sco\al_{r-2},\al$ in the standard order. According to~\eqref{bou}, $\pd\Pi_0\bw=\{\al_1,\al_2\}$. By Lemma~\ref{C_pila},
$i\in\{1,2\}$, $(\la,\chk{\al}_j)=\de_{ij}$ ($j=1\sco r-2$), and $(\la,\chk{\al})=0$. Hence,
$(\la,\chk{\al}_{r-1})+(\la,\chk{\al}_r)\bw=(\la,\chk{\al})\bw=0$, $(\la,\chk{\al}_{r-1})=(\la,\chk{\al}_r)=0$. Thus, $(\la,\chk{\al}_j)=\de_{ij}$ for
any $j=1\sco r$, i.\,e. $\la\bw=\ph_i\bw\in\{\ph_1,\ph_2\}\bw\subs\De\cup\{\ph_1\}$, contradicting~\eqref{C_sug}.
\end{proof}

Applying Corollary~\ref{C_pil}, Lemma~\ref{C_Ct}, and Equations \eqref{bou} and~\eqref{C_sug}, we obtain that
\equ{
\cask{
r\in\{3,4\},&\la\in\{\ph_1+\ph_r,\ph_2+\ph_r\};\\
r=3,&\la=\ph_1+\ph_2+\ph_3;\\
r\in\{4,5\},&\la=\ph_{r-1};\\
r=3,&\la=\ph_3.}}

Indeed, if \eqref{C_con} does not hold, then the claim is evident. Suppose that \eqref{C_con} holds. By Corollary~\ref{C_pil}, $(\la,\chk{\al})=1$ for
any $\al\in\Pi_{\la}$.

Applying Equation~\eqref{bou} and Corollary~\ref{C_pil} for $\Pi':=\{\al_3\sco\al_r\}$, we obtain $(\la,\chk{\al}_k)=0$ for all $k>4$.

Suppose that $(\la,\chk{\al}_r)=0$. By Lemma~\ref{C_Ct}, $\la=\ph_{r-1}$. Hence, $r-1\le4$, $r\le5$. It follows from~\eqref{C_sug} that $r>3$. Thus,
$r\in\{4,5\}$ in this case.

To conclude, we may assume that $(\la,\chk{\al}_r)\ne0$. We have $r\le4$, $r\in\{3,4\}$. By Corollary~\ref{C_pil},
$\{\al_r\}\bw\subs\Pi_{\la}\bw\subs\{\al_1,\al_2,\al_r\}$ and $\br{\Pi_{\la}=\{\al_1,\al_2,\al_r\}}\Ra(r=3)$. Therefore,
$\la\bw\in\{\ph_r,\ph_1\bw+\ph_r,\ph_2\bw+\ph_r\}$, or $r=3$ and $\la=\ph_1+\ph_2+\ph_3$. Note also that, by~\eqref{C_sug}, if $\la=\ph_r$, then $r\ne4$
and thus $r=3$.

\begin{cas} $r=3,4$ and $\la\in\{\ph_1+\ph_r,\ph_2+\ph_r\}$ or $r=3$ and $\la=\ph_1+\ph_2+\ph_3$.
\end{cas}

Using Proposition~\ref{C_all} and Lemma~\ref{nosm}, we come to a~contradiction to the suggestion that $V/G$ is a~smooth manifold.

\begin{cas} $r=4,5$ and $\la=\ph_{r-1}$.
\end{cas}

Applying Proposition~\ref{C_las} and Lemma~\ref{nosm}, we get a~contradiction to the assumption that $V/G$ is a~smooth manifold.

\begin{cas} $r=3$ and $\la=\ph_3$.
\end{cas}

We have $\la\notin Q$ and thus $\de=2$.

By Proposition~\ref{C_thi} and Lemma~\ref{nosm}, $2\cdot\Bn{\bc{\la'\in\Inn{\La}\cln(\la',\chk{\al}_1)\ne0}}\le14+6$,
$\rk_{\Cbb}\br{R(h_{\al_1})}\bw=\Bn{\bc{\la'\in\Inn{\La}\cln\la'(h_{\al_1})\ne0}}\bw=\Bn{\bc{\la'\in\Inn{\La}\cln(\la',\chk{\al}_1)\ne0}}\bw\le10$.

The representation~$R$ of the algebra $\ggt_{\Cbb}=\spg_n(\Cbb)$, $n=6$, is isomorphic to the external cube of its tautological representation, and the
eigenvalues of the semisimple operator $h_{\al_1}\bw\in\spg_n(\Cbb)$ are exactly the numbers $0$, $1$, and~$-1$ with multiplicities $n_0:=n_+:=n_-:=2$.
Hence, $\rk_{\Cbb}\br{R(h_{\al_1})}=\Cn{n}{3}-\Cn{n_0}{1}\cdot\Cn{n_+}{1}\cdot\Cn{n_-}{1}=12$ contradicting $\rk_{\Cbb}\br{R(h_{\al_1})}\le10$.

Thus, we have completely proved Theorem~\ref{C_main} (\name{by contradiction}).

\subsection{Proof of Theorem~\ref{D_main}}

In this subsection, we will prove Theorem~\ref{D_main}.

Suppose that $r>3$ and $\ggt\cong\sog_n$, where $n:=2r\in\N$.

Clearly, $\De\subs\tgt_{\Cbb}^*$ is an indecomposable root system of type~$D_r$.

We will prove Theorem~\ref{D_main} \name{by contradiction}\: assume that $V/G$ is a~smooth manifold and the linear algebra $R(\ggt_{\Cbb})$ is not
isomorphic to any of the linear algebras $\ad(\ggt_{\Cbb})$, $R_{\ph_1}(D_r)$, $R_{2\ph_1}(D_r)$ ($r>3$), and $R_{\ph_r}(D_r)$ ($r=5,6,8$). The latter
means that
\eqn{\label{D_sug}
\begin{array}{lll}
&&\hph{\bigl(\bigr.}\la\notin\De\cup\{\ph_1,2\ph_1\};\\
(r=4)\quad&\Ra\quad&\br{\la\notin\{\ph_3,\ph_4,2\ph_3,2\ph_4\}};\\
\br{r\in\{5,6,8\}}\quad&\Ra\quad&\br{\la\notin\{\ph_{r-1},\ph_r\}}.
\end{array}}
Further, since for any $c_1\sco c_r\in\Z_{\ge0}$ the linear algebra $R_{c_1\ph_1\spl c_r\ph_r}(D_r)$ is isomorphic to the linear algebra
$R_{c_1\ph_1\spl c_{r-2}\ph_{r-2}+c_r\ph_{r-1}+c_{r-1}\ph_r}(D_r)$ and, if $r=4$, to the linear algebra
$R_{c_4\ph_1+c_2\ph_2+c_1\ph_3+c_3\ph_4}(D_4)$, without loss of generality consider that
\eqn{\label{D_perm}
\begin{array}{c}
(\la,\chk{\al}_{r-1})\ge(\la,\chk{\al}_r);\\
\Br{(r=4)\land\br{(\la,\chk{\al}_3)\eqi(\la,\chk{\al}_4)\pmod{2}}}\quad\Ra\quad\br{(\la,\chk{\al}_1)\eqi(\la,\chk{\al}_3)\pmod{2}}.
\end{array}}

\begin{lemma} Let $\Pi'\in\Pc$ be a~system of simple roots such that $\Pi_{\la}\cap\Pi'\ne\es$ and let
$\Pi_0\bw\subs\De\bw\subs\ha{\De}\bw\subs\tgt_{\Cbb}^*$ be the simple system corresponding to the positive root system
$\De^+\bw\cap\ba{\{\la\}\bw\cup\Pi'}\bw\subs\ha{\De}\bw\subs\tgt_{\Cbb}^*$. Suppose that
\eqn{\label{D_con}
\begin{array}{l}
\hph{\bigl(\bigr.}\la\ne\ph_1+\ph_{r-1};\\
(\la=\ph_{r-1})\quad\Ra\quad\Br{\br{\Pi'=\{\al_3\sco\al_r\}\subs\Pi}\ \land\ (r>8)}.
\end{array}}
Then
\begin{nums}{-1}
\item $\Pi_{\la}\cap\Pi'=\{\al\}\subs\Pi$, where $\al\in\pd\Pi'\subs\Pi$, and $(\la,\chk{\al})=1$\~
\item if the simple system $\Pi_0\subs\ha{\De}$ is indecomposable, then \eqref{alb} holds.
\end{nums}
\end{lemma}

\begin{proof} Clearly, $H:=\ba{\{\la\}\cup\Pi'}\subs\ha{\De}\subs\tgt_{\Cbb}^*$ is an $(r-1)$\д dimensional (real) subspace, and, hence, the
intersection of the kernels of all its linear functions has form $\Cbb\xi\subs\tgt_{\Cbb}$, $\xi\in\tgt\sm\{0\}$.

Assume that the claim does not hold.

By Lemma~3.4 in~\cite[\Ss3]{My0}, there exists a~vector $v\in V$ such that $\xi\in\ggt_v$ and $\rk\ggt_v=1$.

It follows from Lemma~\ref{bms} that $\bgm{\Inn{\La}\sm H}\le\de\cdot\bgm{\Inn{\La}\sm H}\le|\De\sm H|+6$.

\begin{cas} $\la\in Q$.
\end{cas}

By Lemma~\ref{D_ev}, $\la\in\De\cup\{2\ph_1\}$ or $r=4$ and $\la\in\{2\ph_3,2\ph_4\}$, contradicting~\eqref{D_sug}.

\begin{cas} $\la\notin Q$ and $(\la,\chk{\al}_{r-1})\eqi(\la,\chk{\al}_r)\pmod{2}$.
\end{cas}

By~\eqref{D_sug}, $\la\notin Q\cup\{\ph_1\}$. Further, according to Lemma~\ref{D_lah}, $r=4$. Thus, $r=4$ and
$(\la,\chk{\al}_3)\eqi(\la,\chk{\al}_4)\pmod{2}$. By~\eqref{D_perm}, $(\la,\chk{\al}_1)\eqi(\la,\chk{\al}_3)\eqi(\la,\chk{\al}_4)\pmod{2}$ implying
$\la\in Q$. This contradicts the assumption.

\begin{cas} $(\la,\chk{\al}_{r-1})\neqi(\la,\chk{\al}_r)\pmod{2}$.
\end{cas}

It is easy to see that $\de=1\in\R$ if $r\in4\Z$ and $\de=2\in\R$ if $r\notin4\Z$, where $\de$ is defined as in the beginning of \Ss\ref{promain}.

By Lemma~\ref{D_odd}, $\la\in\{\ph_{r-1},\ph_r,\ph_1+\ph_{r-1},\ph_1+\ph_r\}$. Further, it follows from \eqref{D_perm} and~\eqref{D_con} that
$\la=\ph_{r-1}$, $\Pi'=(3\sco r)\subs\Pi$, and $r>8$, contradicting Lemma~\ref{D_lah1}.

This completes the proof.
\end{proof}

\begin{imp}\label{D_pila} Let $\Pi'\in\Pc$ be some system of simple roots. Denote by~$\Pi_0$ the simple system corresponding to the positive root system
$\De^+\cap\ba{\{\la\}\cup\Pi'}\subs\ha{\De}\subs\tgt_{\Cbb}^*$. Suppose that \eqref{D_con} holds. Then
\equ{
|\Pi_{\la}\cap\Pi'|\le1,\quad\quad\Pi_{\la}\cap\br{\Pi'\sm(\pd\Pi')}=\es,\quad\quad\quad\fa\al\in\Pi_{\la}\cap\Pi'\quad(\la,\chk{\al})=1.}
If $\Pi_{\la}\cap\Pi'\ne\es$ and the simple system~$\Pi_0$ is indecomposable, then \eqref{alb} holds.
\end{imp}

\begin{imp}\label{D_pil} For any $\al\in\Pi_{\la}$ we have $(\la,\chk{\al})=1$. Besides, on the Dynkin diagram of the simple system~$\Pi$ each two
distinct roots of the subset $\Pi_{\la}\subs\Pi$ correspond to vertices, the path between which contains at least $r-2$ edges.
\end{imp}

\begin{proof} In the case $\la\in\{\ph_{r-1},\ph_1+\ph_{r-1}\}$, the claim is evident, and in the case $\la\bw\notin\{\ph_{r-1},\ph_1\bw+\ph_{r-1}\}$, it
follows from Fact~\ref{cov} and Corollary~\ref{D_pila}.
\end{proof}

\begin{lemma}\label{D_Dt} Suppose that $(\la,\chk{\al}_{r-1})=(\la,\chk{\al}_r)$. Then $(\la,\chk{\al}_1)\seq(\la,\chk{\al}_{r-2})=0$.
\end{lemma}

\begin{proof} By assumption, $\la\notin\{\ph_{r-1},\ph_1+\ph_{r-1}\}$.

Assume that $(\la,\chk{\al}_i)\ne0$ for some $i\in\{1\sco r-2\}$.

The simple system $\Pi':=\{\al_1\sco\al_{r-2}\}\subs\Pi$ satisfies $\Pi'\in\Pc$, $\Pi_{\la}\bw\cap\Pi'\bw\ne\es$, and
$H:=\ba{\{\la\}\cup\Pi'}=\ha{\al_{r-1}-\al_r}^{\perp}$. We have $\al:=\al_{r-2}+\al_{r-1}+\al_r\in\De^+\cap H$,
$\chk{\al}\bw=\ep_{r-2}\bw+\ep_{r-1}$, $\chk{\al}_r=\ep_{r-1}+\ep_r$, $\chk{\al}_{r-1}=\ep_{r-1}-\ep_r$, $\chk{\al}_{r-2}=\ep_{r-2}-\ep_{r-1}$,
$\chk{\al}=\chk{\al}_{r-2}+\chk{\al}_{r-1}+\chk{\al}_r$, $(\la,\chk{\al})=(\la,\chk{\al}_{r-2})+(\la,\chk{\al}_{r-1})+(\la,\chk{\al}_r)$. The positive
root system $\De^+\bw\cap H\bw\subs\ha{\De}\bw\subs\tgt_{\Cbb}^*$ corresponds to the indecomposable simple system~$\Pi_0$ of type~$D_{r-1}$
containing the simple roots $\al_1\sco\al_{r-2},\al$ in the standard order. According to~\eqref{bou}, $\pd\Pi_0\bw\subs\{\al_1,\al_{r-2},\al\}$. By
Corollary~\ref{D_pila}, $i\in\{1,r-2\}$, $(\la,\chk{\al}_j)=\de_{ij}$ ($j=1\sco r-2$), and, also, $(\la,\chk{\al})=0$,
$(\la,\chk{\al}_{r-2})\bw+(\la,\chk{\al}_{r-1})\bw+(\la,\chk{\al}_r)\bw=(\la,\chk{\al})\bw=0$, yielding
$(\la,\chk{\al}_{r-2})\bw=(\la,\chk{\al}_{r-1})\bw=(\la,\chk{\al}_r)\bw=0$, $i\ne r-2$, $i=1$, $(\la,\chk{\al}_j)=\de_{1j}$ ($j=1\sco r$), $\la=\ph_1$,
contradicting~\eqref{D_sug}.
\end{proof}

Applying Corollary~\ref{D_pil}, Lemma~\ref{D_Dt}, and Equations \eqref{D_sug} and~\eqref{D_perm}, we obtain that $\la\bw\in\{\ph_{r-1},\ph_1+\ph_{r-1}\}$
and $(\la=\ph_{r-1})\Ra(r\ne4,5,6,8)$.

\begin{cas} $\la=\ph_1+\ph_{r-1}$.
\end{cas}

It is easy to see that $\de=1\in\R$ if $r\in4\Z$ and $\de=2\in\R$ if $r\notin4\Z$. Using Proposition~\ref{D_all} and Lemma~\ref{nosm}, we get
a~contradiction to the assumption that $V/G$ is a~smooth manifold.

\begin{cas} $r=7$ and $\la=\ph_6$.
\end{cas}

The representation~$R$ of the algebra~$\ggt_{\Cbb}$ is not self-adjoint. Therefore, $\de=2$. Applying Proposition~\ref{D_last} and Lemma~\ref{nosm}, we
come to a~contradiction to the suggestion that $V/G$ is a~smooth manifold.

\begin{cas} $r>8$ and $\la=\ph_{r-1}$.
\end{cas}

The simple system $\Pi':=\{\al_3\sco\al_r\}\subs\Pi$ lies in the family~$\Pc$ and satisfies~\eqref{D_con}. By Corollary~\ref{D_pila},
$\Pi_{\la}\cap\br{\Pi'\sm(\pd\Pi')}=\es$. Since $\Pi'\cong D_{r-2}$ and $r-2\ge7$, we have $\pd\Pi'=\{\al_3\}\subs\Pi$,
$\Pi'\sm(\pd\Pi')=\{\al_4\sco\al_r\}\subs\Pi$. Further, $\Pi_{\la}=\{\al_{r-1}\}\subs\Pi$. Hence,
$\Pi_{\la}\cap\br{\Pi'\sm(\pd\Pi')}=\{\al_{r-1}\}\ne\es$, which is a~contradiction.

Thus, we have completely proved Theorem~\ref{D_main} (\name{by contradiction}).

\section{Orbits of the Weyl group}\label{facts}

In this section, we obtain a number of technical results that play an auxiliary role in \Ss\ref{promain}. More exactly, we should set some quantitative
relations and estimates for weight sets of irreducible representations of complex simple Lie algebras. However, simple linear Lie algebras are not
concerned here since their weight sets can be explicitly described in terms of root systems, root and weight lattices, and Weyl groups. So, we consider
an abstract root system of type $B$, $C$, or~$D$ in a~Euclidean space and orbits of the Weyl group action on the corresponding weight lattice, study in
detail convex hulls of these orbits, and prove the properties needed.

Let $\Eb$ be a~Euclidean space, $\De\subs\Eb$ a~root system, $W\subs\Or(\Eb)$ its Weyl group, and $P\subs\Eb$ and $Q\subs\Eb$ the lattices
$\bc{\la\in\Eb\cln(\la,\chk{\al})\in\Z\,\fa\al\in\De}$ and~$\ha{\De}_{\Z}$ respectively. It is clear that $WP=P$, $WQ=Q$, and $Q\subs P$. Further, let
$\La\subs P$ be some orbit of the action $W\cln P$. Set $\Inn{\La}:=\conv(\La)\cap(\La+Q)\subs P$. We obviously have $W\Inn{\La}=\Inn{\La}$.

\begin{lemma}\label{sxa} If a~subspace $H\subs\Eb$, a~root $\al\in\De\sm H$, and a~finite subset $K\subs\Inn{\La}$ satisfy
$(K-K)\cap\br{(\R\al)\sm\{0\}}=\es$, then $\bgm{\Inn{\La}\sm H}\ge\sums{x\in K}(x,\chk{\al})$.
\end{lemma}

\begin{proof} Let $x\in K$ be an arbitrary vector. We have $\bgm{\Inn{\La}\cap(x+\R\al)}\ge\bgm{(x,\chk{\al})}+1$. By assumption, $\al\notin H$ and thus
$\bgm{H\cap(x+\R\al)}\le1$. Hence, $\Bm{\br{\Inn{\La}\sm H}\cap(x+\R\al)}\bw\ge\bgm{(x,\chk{\al})}\bw\ge(x,\chk{\al})$.

Since $(K-K)\cap\br{(\R\al)\sm\{0\}}=\es$, the subsets $x+\R\al\subs\Eb$, $x\in K$, are pairwise disjoint. Therefore,
$\bgm{\Inn{\La}\sm H}\ge\sums{x\in K}\Bm{\br{\Inn{\La}\sm H}\cap(x+\R\al)}\ge\sums{x\in K}(x,\chk{\al})$.
\end{proof}

Let $r>1$ be some integer.

Consider the Euclidean space~$\R^r$ whose standard basis $\{\ep_i\}_{i=1}^r$ is orthonormal, an indecomposable root system $\De\subs\R^r$ of rank~$r$,
its Weyl group $W\subs\Or(\R^r)$, Weyl chamber $C\subs\R^r$, and the corresponding simple system $\Pi=\{\al_1\sco\al_r\}\subs\De$. Clearly,
$\ha{\De}=\R^r$, the lattices $P\subs\R^r$ and $Q\subs\R^r$ are $W$\д invariant and $Q\subs P$ holds.

Let $\ph_1\sco\ph_r\in P$ be the fundamental weights corresponding to the system of simple roots $\Pi=\{\al_1\sco\al_r\}\subs\De$ keeping the order. Then
$\{\ph_i\}_{i=1}^r$ is a~basis of the lattice $P\subs\R^r$.

For brevity, orbits of the action $W\cln P$ will simply be called \textit{orbits}.

Let $\La\subs P$ be an arbitrary orbit. Set $\Inn{\La}:=\conv(\La)\cap(\La+Q)\subs P$. It is clear that
\begin{nums}{-1}
\renewcommand{\labheadi}[1]{\textup{(#1)}}
\renewcommand{\labhi}[1]{\labheadi{\roman{#1}}}
\item $|\La\cap C|=1$\~
\item $\La-\La\subs Q$\~
\item $\Inn{\La}\sups\La$\~
\item $W\Inn{\La}=\Inn{\La}$\~
\item\label{int} for any orbit $\La'\subs\Inn{\La}$, we have $\Inn{\La}{}'\subs\Inn{\La}$\~
\item\label{intr} for any orbit $\La'\subs\Inn{\La}$, $\La'\ne\La$, we have $\Inn{\La}{}'\cap\La=\es$ (it follows from the fact that $\La$ is the set of
extremal points of~$\Inn{\La}$).
\end{nums}

We will now study the sets~$\Inn{\La}$ separately in the three cases of type $B$, $C$, and~$D$. To that end, fix an orbit $\La\subs P\sm\{0\}$.
Obviously, $\La\cap C=\{\la\}$ for some non-zero $\la\in P\cap C$.

\subsection{The case of type~$B$}

Suppose that $\De=\{\pm\ep_i,\pm\ep_i\pm\ep_j\cln i\ne j\}\subs\R^r$, $C=\{x\in\R^r\cln x_1\sge x_r\ge0\}\bw\subs\R^r$, $\al_i=\ep_i-\ep_{i+1}$
($i=1\sco r-1$), and $\al_r=\ep_r$.

We have $\De\cong B_r$, $|\De|=2r^2$, $|P/Q|=2$. Let $\la_1\sco\la_r$ be the components of the vector $\la\in\La\cap C$. For some
$\la_0\in\BC{0,\frac{1}{2}}$, we have $\{\la_1\sco\la_r\}\subs\la_0+\Z_{\ge0}$.

It is easy to see that
\begin{nums}{-1}
\item\label{zth} if $\la=\ph_r$, then $\Inn{\La}=\La$ and $\bgm{\Inn{\La}}=|\La|=2^r$\~
\item\label{fst} if $\la=\ep_1+\ph_r$, then $\bgm{\Inn{\La}}=2^r(r+1)$ (follows from \ref{zth} and items \ref{int} and~\ref{intr})\~
\item\label{snd} if $\la=(\ep_1+\ep_2)+\ph_r$, then $\bgm{\Inn{\La}}=2^r\cdot\bbr{\Cn{r}{2}+r+1}$ (follows from \ref{fst} and items \ref{int}
and~\ref{intr})\~
\item\label{ze} if $\la_0=\frac{1}{2}$, then $\bgm{\Inn{\La}}\ge2^r$ (follows from \ref{zth} and items \ref{int} and~\ref{intr})\~
\item\label{on} if $\la_0=\frac{1}{2}$ and $\la\ne\ph_r$, then $\bgm{\Inn{\La}}\ge2^r(r+1)$ (follows from \ref{fst} and items \ref{int} and~\ref{intr})\~
\item\label{tw} if $\la_0=\frac{1}{2}$ and $\la\ne\ph_r,\ep_1+\ph_r$, then $\bgm{\Inn{\La}}\ge2^r\cdot\bbr{\Cn{r}{2}+r+1}$ (follows from \ref{snd} and
items \ref{int} and~\ref{intr}).
\end{nums}

\begin{lemma}\label{B_arb} If $r>2$, $\la_0=\frac{1}{2}$, $\la\ne\ph_r$, and $\bgm{\Inn{\La}}\le4(r^2+3)$, then $r=3$ and $\la=\ep_1+\ph_3$.
\end{lemma}

\begin{proof} We have $2^r(r+1)\le\bgm{\Inn{\La}}\le4(r^2+3)$ and thus $r<4$, $r=3$. If $\la\ne\ep_1+\ph_3$, then
$\bgm{\Inn{\La}}\ge2^r\cdot\bbr{\Cn{r}{2}+r+1}=56>48=4(r^2+3)$, which contradicts the condition.
\end{proof}

\begin{prop}\label{B_larb} Let $H\subs\R^r$ be an arbitrary hyperplane. If $\la\notin\De\cup\{2\ph_1\}$, then the inequality
$\bgm{\Inn{\La}\sm H}\le|\De\sm H|+6$ can take place just in the following cases\:
\begin{nums}{-1}
\item $r>2$, $\la=\ph_r$, and $2^{r-1}\le|\De\sm H|+6$\~
\item $r=3$ and $\la\in\{2\ph_3,\ph_1+\ph_3\}$\~
\item $r=2$ and $\la\in\ph_2+Q$.
\end{nums}
\end{prop}

\begin{proof} Assume that $\bgm{\Inn{\La}\sm H}\le|\De\sm H|+6$, $\la\notin\De$, and $\la\ne2\ph_1$.

If $\la_0=0$, then, since $\la\notin\De\cup\{2\ph_1\}$, we have $\la_1\spl\la_r\ge3$, implying $\ph_2=\ep_1+\ep_2\in\Inn{\La}$, $\De\subs\Inn{\La}$,
$\Inn{\La}\sm H=(\De\sm H)\sqcup\br{\Inn{\La}\sm(\De\cup H)}$, $\bgm{\Inn{\La}\sm(\De\cup H)}=\bgm{\Inn{\La}\sm H}-|\De\sm H|\le6$.

Since $H\ne\R^r$, there exists a~number $p\in\{1\sco r\}$ such that $\ep_p\notin H$.

Let $\ga\in W$ be the reflection against the hyperplane $(\R\ep_p)^{\perp}\subs\R^r$ and $\Ga\subs W$ the subgroup $\{E,\ga\}$, where $E$ denotes the
identity operator. If $x\in\R^r$ and $x,\ga x\in H$, then $x-\ga x\in\R\ep_p\cap H=0$, i.\,e. $x=\ga x$. Hence, if $K\subs\R^r$, $\Ga K=K$, and
$K\cap(\R\ep_p)^{\perp}=\es$, then $K^{\Ga}=\es$, $|K\sm H|\ge|K/\Ga|=\frac{1}{2}\cdot|K|$.

\begin{cas}\label{B_ev} $\la_0=0$.
\end{cas}

We have $\la_1\spl\la_r\ge3$ and $\bgm{\Inn{\La}\sm(\De\cup H)}\le6$. It follows that
\begin{align*}
(r\ge3)\quad&\Ra\quad\br{\ep_1+\ep_2+\ep_3\in\Inn{\La}};\\
(r=3)\quad&\Ra\quad\Br{\br{2\ep_1+\ep_2\in\Inn{\La}}\lor(\la=\ep_1+\ep_2+\ep_3)};\\
(r=2)\quad&\Ra\quad\br{2\ep_1+\ep_2\in\Inn{\La}}.
\end{align*}

Suppose that $r\ge4$. Then $\ep_1+\ep_2+\ep_3\in\Inn{\La}$. For the subset $K\subs P$ of all vectors $\pm\ep_i\pm\ep_j\pm\ep_p\in P$, where
$p$ is fixed as above, $i,j\in\{1\sco r\}\sm\{p\}$ and $i\ne j$, we have $|K|=4(r-1)(r-2)\ge24$, $K\subs\Inn{\La}\sm\De$, $\Ga K=K$, and
$K\cap(\R\ep_p)^{\perp}=\es$. Hence, $\bgm{\Inn{\La}\sm(\De\cup H)}\bw\ge|K\sm H|\bw\ge\frac{1}{2}\cdot|K|\bw\ge12$, which contradicts the inequality
$\bgm{\Inn{\La}\sm(\De\cup H)}\le6$.

Assume that $r=3$ and $\la\ne\ep_1+\ep_2+\ep_3$. Then $2\ep_1+\ep_2\in\Inn{\La}$ and $\ep_1+\ep_2+\ep_3\in\Inn{\La}$. For the subset $K\subs P$ of all
vectors $\pm\ep_1\pm\ep_2\pm\ep_3\in P$ and $\pm2\ep_i\pm\ep_p\in P$ ($i\in\{1,2,3\}\sm\{p\}$) we have $|K|=16$, $K\subs\Inn{\La}\sm\De$, $\Ga K=K$, and
$K\cap(\R\ep_p)^{\perp}=\es$. It follows that $\bgm{\Inn{\La}\sm(\De\cup H)}\bw\ge|K\sm H|\bw\ge\frac{1}{2}\cdot|K|\bw=8$ contradicting
$\bgm{\Inn{\La}\sm(\De\cup H)}\le6$.

Suppose that $r=2$. Then $2\ep_1+\ep_2\in\Inn{\La}$ and $2\ep_1\in\Inn{\La}$. For the subset $K\subs P$ of all vectors $\pm2\ep_i+\de\ep_j\in P$
($\{i,j\}=\{1,2\}$, $\de\in\{0,\pm1\}$) we have $|K|=12$ and $K\subs\Inn{\La}\sm\De$. Also, $\dim H=r-1=1$, and the subset $K\subs P$ does not contain
two proportional vectors of different lengths. Thus, $|K\cap H|\le2$, $\bgm{\Inn{\La}\sm(\De\cup H)}\ge|K\sm H|\ge|K|-2=10$, contradicting
the inequality $\bgm{\Inn{\La}\sm(\De\cup H)}\le6$.

Now, we can conclude that $r=3$ and $\la=\ep_1+\ep_2+\ep_3=2\ph_3$.

\begin{cas}\label{B_odd} $r>2$ and $\la_0=\frac{1}{2}$.
\end{cas}

Clearly, $\Ga\Inn{\La}=\Inn{\La}$ and $\Inn{\La}\cap(\R\ep_p)^{\perp}=\es$. Therefore, $\bgm{\Inn{\La}\sm H}\ge\frac{1}{2}\cdot\bgm{\Inn{\La}}$.
Hence, $\frac{1}{2}\cdot\bgm{\Inn{\La}}\bw\le\bgm{\Inn{\La}\sm H}\bw\le|\De\sm H|+6\bw\le|\De|+6\bw=2r^2+6$, $\bgm{\Inn{\La}}\le4(r^2+3)$.

By Lemma~\ref{B_arb}, one of the following situations occurs\:
\begin{nums}{-1}
\item $r>2$ and $\la=\ph_r$ \ter{consequently, $\bgm{\Inn{\La}}=2^r$, $2^{r-1}=\frac{1}{2}\cdot\bgm{\Inn{\La}}\le|\De\sm H|+6$}\~
\item $r=3$ and $\la=\ep_1+\ph_3=\ph_1+\ph_3$.
\end{nums}

\begin{cas} $r=2$ and $\la_0=\frac{1}{2}$.
\end{cas}

We have $\la\in\ph_2+(\Z\ep_1\oplus\Z\ep_2)=\ph_2+Q$.

This completely proves the claim.
\end{proof}

\begin{lemma}\label{B_lah2} Let $H\subs\R^r$ be some hyperplane. If $r=2$, $\la\in\ph_2+Q$, and $2\cdot\bgm{\Inn{\La}\sm H}\bw\le|\De\sm H|+6$, then
$\la=\ph_2$.
\end{lemma}

\begin{proof} We have $\la\in\ph_2+Q=\ph_2+(\Z\ep_1\oplus\Z\ep_2)$, $\la_1,\la_2\in\frac{1}{2}+\Z_{\ge0}$. Further, $|\De|=2r^2=8$,
$\bgm{\Inn{\La}\sm H}\le\frac{1}{2}\cdot\br{|\De|+6}=7$.

Suppose that $\la\ne\ph_2$. Then $\la_1\ge\frac{3}{2}$ and thus $\frac{1}{2}\cdot(3\ep_1+\ep_2)\in\Inn{\La}$. The subset $K\subs P$ of all vectors
$\frac{1}{2}\cdot(\pm\ep_1\pm\ep_2)\in P$ and $\frac{1}{2}\cdot(\pm3\ep_i\pm\ep_j)\in P$ ($\{i,j\}=\{1,2\}$) satisfies $|K|=12$ and $K\subs\Inn{\La}$.
Also, $\dim H=r-1=1$, and the subset $K\subs P$ does not contain two proportional vectors of different lengths. It follows that $|K\cap H|\le2$,
$\bgm{\Inn{\La}\sm H}\ge|K\sm H|\ge|K|-2=10$, contradicting $\bgm{\Inn{\La}\sm H}\le7$.
\end{proof}

Set $\Pi_{\la}:=\bc{\al\in\Pi\cln(\la,\chk{\al})\ne0}\subs\Pi$. Let $\Pi'\subs\Pi\subs\R^r$ be an indecomposable system of simple roots of order $r-2$
such that
\eqn{\label{B_nes}
(r>2)\quad\Ra\quad(\Pi_{\la}\cap\Pi'\ne\es).}

\begin{lemma}\label{B_lah} Set $H:=\ba{\{\la\}\cup\Pi'}\subs\R^r$. If $\la\notin\De$ and $\la\ne2\ph_1$, then the inequality
$\bgm{\Inn{\La}\sm H}\le|\De\sm H|+6$ can take place just in the following cases\:
\begin{nums}{-1}
\item $r=3,4,5,6$ and $\la=\ph_r$\~
\item $r=3$ and $\la\in\{2\ph_3,\ph_1+\ph_3\}$\~
\item $r=2$ and $\la\in\ph_2+Q$.
\end{nums}
\end{lemma}

\begin{proof} Assume that $\bgm{\Inn{\La}\sm H}\le|\De\sm H|+6$, $\la\notin\De$, and $\la\ne2\ph_1$.

It is easy to see that $\dim H=r-1$. By Proposition~\ref{B_larb}, it suffices to prove that
\equ{
\Br{(\la=\ph_r)\land\br{2^{r-1}\le|\De\sm H|+6}}\quad\Ra\quad(r<7).}

Suppose that $\la=\ph_r$ and $2^{r-1}\le|\De\sm H|+6$. Then $\Pi_{\la}=\{\al_r\}\subs\Pi$. According to~\eqref{B_nes}, $\Pi'=\{\al_3\sco\al_r\}$,
$\Pi'\cong B_{r-2}$, $|\De\cap H|\ge\bgm{\De\cap\ha{\Pi'}}=2(r-2)^2=|\De|-(8r-8)$. Hence, $|\De\sm H|\le8r-8$, $2^{r-1}\le|\De\sm H|+6\le8r-2$
and thus $r<7$.
\end{proof}

\begin{prop}\label{B_all} If $r>2$ and $\la\in\{\ph_1+\ph_r,\ph_2+\ph_r,\ph_1+\ph_2+\ph_r\}$, then
\begin{nums}{-1}
\item $2\la\notin\De\cup(\De+\De)$\~
\item there exist a~subset $\Om\subs\La$ and a~hyperplane $H\subs\R^r$ such that $\ha{\Om}=H$, $(\Om+\Om)\cap\De\bw=(\Om-\Om)\cap\De\bw=\es$, and
$\bgm{\Inn{\La}\sm H}>|\De\sm H|+6$.
\end{nums}
\end{prop}

\begin{proof} Firstly, note that $\la_1\ge\frac{3}{2}$, $2\la_1\ge3$, and, consequently, $2\la\notin\De\cup(\De+\De)$, which proves the first statement.
Let us pass to proving the second.

It is clear that $c\ep_1+\ep_2+\ph_r\in\La$, where
\equ{
c:=\case{
0,&\la=\ph_1+\ph_r;\\
1,&\la=\ph_2+\ph_r;\\
2,&\la=\ph_1+\ph_2+\ph_r.}}

Set $\Om:=\{c\ep_1-2\ep_j+\ph_r\cln j=2\sco r\}\subs\La$ and $H:=\ha{\Om}\subs\R^r$.

We have $H=\bc{x\in\R^r\cln(r-5)x_1=(2c+1)(x_2\spl x_r)}\subs\R^r$ and $|\Om|=\dim H=r-1$. Also, $(\Om+\Om)\cap\De=(\Om-\Om)\cap\De=\es$.

It remains to prove that $\bgm{\Inn{\La}\sm H}>|\De\sm H|+6$.

Assume that $\bgm{\Inn{\La}\sm H}\le|\De\sm H|+6$. By Proposition~\ref{B_larb}, $r=3$ and $\la=\ph_1+\ph_3$. Therefore, $c=0$,
$\bgm{\Inn{\La}}=2^r(r+1)=32$, $|\De|=2r^2=18$, and $H=\bc{x\in\R^r\cln -2x_1=x_2+x_3}\subs\R^r$. We have
$\Inn{\La}\cap H=\Om\sqcup(-\Om)\sqcup\bc{\pm(\ph_3-\ep_1)}\subs\Inn{\La}$, $\bgm{\Inn{\La}\cap H}=2r=6$,
$\bgm{\Inn{\La}\sm H}\bw=26\bw>|\De|+6\bw\ge|\De\sm H|+6$.

Thus, we come to a~contradiction, proving the claim.
\end{proof}

\begin{prop}\label{B_last} If $r=6$ and $\la=\ph_6$, then there exist a~subset $\Om\subs\La$ and a~hyperplane $H\subs\R^6$ such that $\ha{\Om}=H$,
$(\Om-\Om)\cap\De=\es$, and $2\cdot\bgm{\Inn{\La}\sm H}>|\De\sm H|+6$.
\end{prop}

\begin{proof} By assumption, $|\La|=2^r=64$ and $|\De|=2r^2=72$.

Set $\Om:=\bc{-\ph_6,\ph_6-(\ep_1+\ep_2),\ph_6-(\ep_3+\ep_4),\ph_6-\ep_5,\ph_6-(\ep_2+\ep_3+\ep_6)}\subs\La$. We have $(\Om-\Om)\cap\De=\es$ and
$H:=\ha{\Om}=\bc{x\in\R^6\cln x_1+x_3=x_2+x_4}\subs\R^6$. Consequently, $|\La\cap H|=24$, $|\La\sm H|=40$, $2\cdot|\La\sm H|=80>|\De|+6\ge|\De\sm H|+6$.
\end{proof}

\begin{prop}\label{B_thi} If $r=4$ and $\la=\ph_3$, then
\begin{nums}{-1}
\item $2\la\notin\De\cup(\De+\De)$\~
\item there exist a~subset $\Om\subs\La$ and a~hyperplane $H\subs\R^4$ such that $\ha{\Om}=H$, $(\Om+\Om)\cap\De\bw=(\Om-\Om)\cap\De\bw=\es$, and
$\bgm{\Inn{\La}\sm H}>|\De\sm H|+6$.
\end{nums}
\end{prop}

\begin{proof} We have $\la=\ep_1+\ep_2+\ep_3$. It follows that $2\la=2(\ep_1+\ep_2+\ep_3)\notin\De\cup(\De+\De)$.

Set $\Om:=\{\ph_3-2\ep_j\cln j=1,2,3\}\subs\La$. It is clear that $(\Om+\Om)\cap\De=(\Om-\Om)\cap\De=\es$, $H:=\ha{\Om}=\ha{\ep_1,\ep_2,\ep_3}\subs\R^4$,
and $\dim H=3$. By Proposition~\ref{B_larb}, $\bgm{\Inn{\La}\sm H}>|\De\sm H|+6$.
\end{proof}

\begin{lemma}\label{B_phr} If $\la=\ph_r$, then $\la\notin Q$ and $2\la\in Q$.
\end{lemma}

We omit the proof since it is evident.

\subsection{The case of type~$C$}

Suppose that $r>2$, $\De=\{\pm2\ep_i,\pm\ep_i\pm\ep_j\cln i\ne j\}\subs\R^r$, $C=\{x\in\R^r\cln x_1\sge x_r\ge0\}\bw\subs\R^r$, $\al_i=\ep_i-\ep_{i+1}$
($i=1\sco r-1$), and $\al_r=2\ep_r$.

We have $\De\cong C_r$, $|\De|=2r^2$, and $|P/Q|=2$.

Set $\de:=1\in\R$ if $\la\in Q$ and $\de:=2\in\R$ if $\la\notin Q$.

\begin{lemma}\label{C_larb} Let $H\subs\R^r$ be some hyperplane. If $\la\notin\De\cup\{\ph_1\}$ and $\de\cdot\bgm{\Inn{\La}\sm H}\bw\le|\De\sm H|+6$,
then at least one of the following conditions \eqref{C_con1} and~\eqref{C_con2} holds\:
\begin{gather}
r=4,\quad\quad\quad\la=\ph_4;\label{C_con1}\\
\la=\ph_k\quad(k\in3+2\Z_{\ge0},\ k\le r),\quad\quad|\De\sm H|+6\ge4(r-1)(r-2).\label{C_con2}
\end{gather}
\end{lemma}

\begin{proof} By assumption, $\la_1\sco\la_r\in\Z_{\ge0}$ and $\la_0:=\la_1\spl\la_r\in3+\Z_{\ge0}$.

Since $H\ne\R^r$, there exists a~number $p\in\{1\sco r\}$ such that $\ep_p\notin H$.

Let $\ga\in W$ be the reflection against the hyperplane $(\R\ep_p)^{\perp}\subs\R^r$ and $\Ga\subs W$ the subgroup $\{E,\ga\}$, where $E$ denotes the
identity operator. If $x\in\R^r$ and $x,\ga x\in H$, then $x-\ga x\in\R\ep_p\cap H=0$, i.\,e. $x=\ga x$. Hence, if $K\subs\R^r$, $\Ga K=K$, and
$K\cap(\R\ep_p)^{\perp}=\es$, then $K^{\Ga}=\es$, $|K\sm H|\ge|K/\Ga|=\frac{1}{2}\cdot|K|$.

\begin{cas}\label{C_ev} $\la\in Q$, $\de=1$.
\end{cas}

The number $\la_0\in\Z$ is even. Therefore, $\la_0\in4+2\Z_{\ge0}$ and $(\la_1>1)\Ra\br{2\ep_1+\ep_2+\ep_3\in\Inn{\La}}$. Also,
$\bgm{\Inn{\La}\sm H}\le|\De\sm H|+6$, $\bgm{\Inn{\La}\sm(\De\cup H)}\le\Bm{\De\bbl\br{\Inn{\La}\cup H}}+6\le\bgm{\De\sm\Inn{\La}}+6$.

Assume that $\la_1>1$.

We have $2\ep_1+\ep_2+\ep_3\in\Inn{\La}$, $2\ph_1=2\ep_1\in\Inn{\La}$, $\De\subs\Inn{\La}$, $\bgm{\Inn{\La}\sm(\De\cup H)}\le6$. The subset $K\subs P$ of
all vectors $\pm2\ep_i\pm\ep_j\pm\ep_p\in P$, where $i,j\in\{1\sco r\}\sm\{p\}$ and $i\ne j$, satisfies $|K|=8(r-1)(r-2)\ge16$, $K\subs\Inn{\La}\sm\De$,
$\Ga K=K$, and $K\cap(\R\ep_p)^{\perp}=\es$. Hence, $\bgm{\Inn{\La}\sm(\De\cup H)}\bw\ge|K\sm H|\bw\ge\frac{1}{2}\cdot|K|\bw\ge8\bw>6$, which is
a~contradiction.

Consequently, $\la_1=1$.

Let $\De_{\max}$ (resp.~$\De_{\min}$) be the subset of long (resp. short) roots of the root system $\De\subs\R^r$. Then $\De=\De_{\max}\sqcup\De_{\min}$,
$\De_{\min}=W(\ep_1+\ep_2)$, and $|\De_{\max}|=2r$.

Set $r_0:=r-2\in\N$.

Since $\la_1=1$, we have $\la=\ep_1\spl\ep_k=\ph_k$, $k\in\{1\sco r\}$. Note that $k=\la_0\in4+2\Z_{\ge0}$. Therefore, $r\ge k\ge4$ and, also,
$\ep_1+\ep_2+\ep_3+\ep_4\in\Inn{\La}$, $\ep_1+\ep_2\in\Inn{\La}$, $\De_{\min}\subs\Inn{\La}$, $\De\bbl\Inn{\La}\subs\De_{\max}$,
$\bgm{\Inn{\La}\sm(\De\cup H)}\le\bgm{\De\sm\Inn{\La}}+6\le|\De_{\max}|+6=2r+6$. The subset $K\subs P$ of all vectors
$\pm\ep_{i_1}\pm\ep_{i_2}\pm\ep_{i_3}\pm\ep_p\in P$, where $i_1,i_2,i_3\in\{1\sco r\}\sm\{p\}$ are pairwise distinct numbers, satisfies
$|K|=\frac{8}{3}\cdot(r-1)(r-2)(r-3)=\frac{8}{3}\cdot(r_0^3-r_0)$, $K\subs\Inn{\La}\sm\De$, $\Ga K=K$, and $K\cap(\R\ep_p)^{\perp}=\es$. Consequently,
$2r+6\ge\bgm{\Inn{\La}\sm(\De\cup H)}\ge|K\sm H|\ge\frac{1}{2}\cdot|K|=\frac{4}{3}\cdot(r_0^3-r_0)$, $\frac{2}{3}\cdot(r_0^3-r_0)\le r+3=r_0+5$,
$2(r_0^3-r_0)\le3(r_0+5)$, $(2r_0^2-5)r_0\le15$, and thus $r_0\le2$, $r\le4$. At the same time, as said above, $r\ge k\ge4$ and $\la=\ph_k$,
implying~\eqref{C_con1}.

\begin{cas}\label{C_odd} $\la\notin Q$, $\de=2$.
\end{cas}

The number $\la_0\in\Z$ is odd. Consequently, $\la_0\in3+2\Z_{\ge0}$, $\ep_1+\ep_2+\ep_3\in\Inn{\La}$, and, also,
$(\la_1>1)\Ra\br{2\ep_1+\ep_2\in\Inn{\La}}$. Further, $2\cdot\bgm{\Inn{\La}\sm H}\le|\De\sm H|+6\le|\De|+6=2r^2+6$, implying
$\bgm{\Inn{\La}\sm H}\le r^2+3$.

Assume that $\la_1>1$.

We have $2\ep_1+\ep_2\in\Inn{\La}$, $\ep_1+\ep_2+\ep_3\in\Inn{\La}$, $\ep_1\in\Inn{\La}$. The subset $K\subs P$ of all vectors
$\pm\ep_i\pm\ep_j\pm\ep_p\in P$, $\pm2\ep_i\pm\ep_p\in P$, $\pm\ep_i\pm2\ep_p\in P$, and $\pm\ep_p\in P$ ($i,j\in\{1\sco r\}\sm\{p\}$, $i\ne j$)
satisfies $|K|=4r(r-1)+2$, $K\subs\Inn{\La}$, $\Ga K=K$, and $K\cap(\R\ep_p)^{\perp}=\es$. Therefore,
$\bgm{\Inn{\La}\sm H}\ge|K\sm H|\ge\frac{1}{2}\cdot|K|>2r(r-1)=(r^2+3)+(r+1)(r-3)\ge r^2+3$ contradicting $\bgm{\Inn{\La}\sm H}\le r^2+3$.

Hence, $\la_1=1$.

We have $\la=\ep_1\spl\ep_k=\ph_k$, $k\in\{1\sco r\}$, $k=\la_0\in3+2\Z_{\ge0}$, and $\ep_1+\ep_2+\ep_3\in\Inn{\La}$.
Let $K\subs P$ be the subset of all vectors $\pm\ep_i\pm\ep_j\pm\ep_p\in P$, where $i,j\in\{1\sco r\}\sm\{p\}$ and $i\ne j$. Then
$|K|=4(r-1)(r-2)$, $K\subs\Inn{\La}$, $\Ga K=K$, and $K\cap(\R\ep_p)^{\perp}=\es$. Therefore, $\bgm{\Inn{\La}\sm H}\ge|K\sm H|\ge\frac{1}{2}\cdot|K|$,
$|\De\sm H|+6\ge2\cdot\bgm{\Inn{\La}\sm H}\ge|K|=4(r-1)(r-2)$, implying~\eqref{C_con2}.
\end{proof}

Let $\Pi'\subs\Pi\subs\R^r$ be an indecomposable simple system of order $r-2$.

\begin{lemma}\label{C_lah} Set $H:=\ba{\{\la\}\cup\Pi'}\subs\R^r$. Suppose that $\la\notin\De\cup\{\ph_1\}$ and
$\de\cdot\bgm{\Inn{\La}\sm H}\bw\le|\De\sm H|+6$. Then $r\in\{3,4\}$ and $\la=\ph_k$ \ter{$k\in\N$, $3\le k\le r$}.
\end{lemma}

\begin{proof} Clearly, $\dim H=r-1$. By Lemma~\ref{C_larb}, it suffices to prove that \eqref{C_con2} implies $r<5$.

Assume that \eqref{C_con2} holds.

The indecomposable simple system $\Pi'\subs\Pi\subs\R^r$ of order $r-2$ corresponds to an indecomposable root system of rank $r-2$ and of order at
least $(r-1)(r-2)$. Consequently, $|\De\cap H|\ge\bgm{\De\cap\ha{\Pi'}}\ge(r-1)(r-2)$. Further, according to~\eqref{C_con2},
$|\De\sm H|+6\ge4(r-1)(r-2)$. Hence, $5(r-1)(r-2)\le|\De|+6=2r^2+6$, $3r(r-5)+4\le0$, $r<5$.
\end{proof}

\begin{prop}\label{C_all} If $\la\in\{\ph_1+\ph_r,\ph_2+\ph_r,\ph_1+\ph_2+\ph_r\}$, then
\begin{nums}{-1}
\item $2\la\notin\De\cup(\De+\De)$\~
\item there exist a~subset $\Om\subs\La$ and a~hyperplane $H\subs\R^r$ such that $\ha{\Om}=H$, $(\Om+\Om)\cap\De\bw=(\Om-\Om)\cap\De\bw=\es$, and
$\de\cdot\bgm{\Inn{\La}\sm H}>|\De\sm H|+6$.
\end{nums}
\end{prop}

\begin{proof} By assumption, $\la_1\sge\la_r\ge1$, $\la_1\spl\la_r\ge r>2$, $2(\la_1\spl\la_r)>4$, and thus $2\la\notin\De\cup(\De+\De)$.

This proves the first statement. Let us pass to proving the second.

It is obvious that $c\ep_1+\ep_2+\ph_r\in\La$, where
\equ{
c:=\case{
0,&\la=\ph_1+\ph_r;\\
1,&\la=\ph_2+\ph_r;\\
2,&\la=\ph_1+\ph_2+\ph_r.}}

Set $\Om:=\{c\ep_1-3\ep_j+\ph_r\cln j=2\sco r\}\subs\La$ and $H:=\ha{\Om}\subs\R^r$.

We have $H=\bc{x\in\R^r\cln(r-4)x_1=(c+1)(x_2\spl x_r)}\subs\R^r$ and $|\Om|=\dim H=r-1$. Also, $(\Om+\Om)\cap\De=(\Om-\Om)\cap\De=\es$.

Finally, by Lemma~\ref{C_larb}, $\de\cdot\bgm{\Inn{\La}\sm H}>|\De\sm H|+6$.
\end{proof}

\begin{prop}\label{C_las} Suppose that $r\in\{4,5\}$ and $\la=\ph_{r-1}$. Then
\begin{nums}{-1}
\item $2\la\notin\De\cup(\De+\De)$\~
\item there exist a~subset $\Om\subs\La$ and a~hyperplane $H\subs\R^4$ such that $\ha{\Om}=H$, $(\Om-\Om)\cap\De=\es$,
$(\de=1)\Ra\br{(\Om+\Om)\cap\De=\es}$, and $\de\cdot\bgm{\Inn{\La}\sm H}>|\De\sm H|+6$.
\end{nums}
\end{prop}

\begin{proof} We have $\la=\ep_1\spl\ep_{r-1}$, $\la_1\spl\la_r=r-1>2$, $2(\la_1\spl\la_r)>4$, and thus $2\la\notin\De\cup(\De+\De)$.

Set $\Om:=\{\ph_{r-1}-2\ep_j\cln j=1\sco r-1\}\subs\La$.

Note that $(\Om-\Om)\cap\De=\es$ and
\eqn{\label{C_sic}
\fa\la'\in\Om\quad\quad\quad\la'_1\spl\la'_r=(r-1)-2=r-3>0.}
We have $H:=\ha{\Om}=\ha{\ep_1\sco\ep_{r-1}}\subs\R^r$. It follows that $\dim H=r-1$ and, also,
$\De\sm H\bw=\{\pm\ep_j\pm\ep_r,\pm2\ep_r\cln j=1\sco r-1\}\bw\subs\De$, $|\De\sm H|=4r-2<4r\le4r(r-3)$, $|\De\sm H|+6\bw<4r(r-3)+6\bw<4(r-1)(r-2)$.

Lemma~\ref{C_larb} and the relations $|\De\sm H|+6<4(r-1)(r-2)$ and $\la\notin\De\cup\{\ph_1,\ph_r\}$ imply $\de\cdot\bgm{\Inn{\La}\sm H}>|\De\sm H|+6$.

It remains to prove that $(\de=1)\Ra\br{(\Om+\Om)\cap\De=\es}$.

If $\de=1$, then $\ph_{r-1}\in Q$, $r-1\in2\Z$, $r=5$, $2(r-3)>2$, and, by~\eqref{C_sic}, $(\Om+\Om)\cap\De=\es$.
\end{proof}

\begin{prop}\label{C_thi} Suppose that $r=3$ and $\la=\ph_3$. Denote by~$H$ the hyperplane $(\R\al_1)^{\perp}\subs\R^3$. Then $|\De\sm H|=14$, and,
also, there exists a~subset $\Om\subs\La$ such that $\ha{\Om}=H$ and $(\Om-\Om)\cap\De=\es$.
\end{prop}

\begin{proof} By assumption, $\la=\ep_1+\ep_2+\ep_3$, $|\De|=2r^2=18$, and $H=\bc{x\in\R^3\cln x_1=x_2}\subs\R^3$.

We have $\De\cap H=\bc{\pm(\ep_1+\ep_2),\pm2\ep_3}\subs\De$, $|\De\cap H|=4$, $|\De\sm H|=14$, and the subset $\Om:=\bc{\pm(\ep_1+\ep_2)+\ep_3}\subs\La$
satisfies $\ha{\Om}=H$ and $(\Om-\Om)\cap\De=\es$.
\end{proof}

\subsection{The case of type~$D$}

Suppose that $r>3$, $\De=\{\pm\ep_i\pm\ep_j\cln i\ne j\}\subs\R^r$, $C=\{x\in\R^r\cln x_1\sge x_{r-1}\bw\ge|x_r|\}\bw\subs\R^r$, $\al_i=\ep_i-\ep_{i+1}$
($i=1\sco r-1$), and $\al_r=\ep_{r-1}+\ep_r$.

We have $\De\cong D_r$, $|\De|=2r(r-1)$, $|P/Q|=4$, $\la_1\sco\la_{r-1}\in\la_0+\Z_{\ge0}$, and $\la_r\bw\in\la_0+\Z$ ($\la_0\in\BC{0,\frac{1}{2}}$).

Let $p,q\in\{1\sco r\}$ be some distinct numbers, $\pi$ the orthogonal projection operator $\R^r\thra\ha{\ep_p,\ep_q}$, $\al$ the root
$\ep_p+\ep_q\in\De$, and $K_0\subs\ha{\ep_p,\ep_q}\subs\R^r$ the subset of all vectors $\frac{1}{2}\cdot(3\ep_i\pm\ep_j)$ ($\{i,j\}=\{p,q\}$). Set
$K:=\Inn{\La}\cap\pi^{-1}(K_0)\subs\Inn{\La}$.

It is easy to see that
\begin{nums}{-1}
\item $(K-K)\cap\br{(\R\al)\sm\{0\}}=\es$\~
\item we have $\sums{x\in K_0}(x_p+x_q)=6$ where every summand is positive\~
\item if $\la\in\ep_1+\{\ph_{r-1},\ph_r\}$, then $\bgm{\Inn{\La}\cap\pi^{-1}(x)}=2^{r-3}$ for any $x\in K_0$\~
\item if $\la\in\ep_1+\ep_2+\{\ph_{r-1},\ph_r\}$, then $\bgm{\Inn{\La}\cap\pi^{-1}(x)}=2^{r-3}(r-1)$ for any $x\in K_0$\~
\item if $\la_0=\frac{1}{2}$ and $\la\ne\ph_{r-1},\ph_r$, then $\bgm{\Inn{\La}\cap\pi^{-1}(x)}\ge2^{r-3}$ for any $x\in K_0$, and, consequently,
$\sums{x\in K}(x,\chk{\al})=\sums{x\in K}(x_p+x_q)\ge6\cdot2^{r-3}$\~
\item if $\la_0=\frac{1}{2}$, $\la\ne\ph_{r-1},\ph_r$, and $\la\ne\ep_1+\ph_{r-1},\ep_1+\ph_r$, then $\bgm{\Inn{\La}\cap\pi^{-1}(x)}\ge2^{r-3}(r-1)$ for
any $x\in K_0$, and thus $\sums{x\in K}(x,\chk{\al})=\sums{x\in K}(x_p+x_q)\ge6\cdot2^{r-3}(r-1)$.
\end{nums}

\begin{lemma}\label{D_odd} Suppose that $(\la,\chk{\al}_{r-1})\neqi(\la,\chk{\al}_r)\pmod{2}$ and $\la\ne\ph_{r-1},\ph_r$. Set $\de:=1\bw\in\R$ if
$r\in4\Z$ and $\de:=2\in\R$ if $r\notin4\Z$. If some hyperplane $H\subs\R^r$ satisfies $\de\cdot\bgm{\Inn{\La}\sm H}\bw\le|\De\sm H|+6$, then $r=4$ and
$\la\in\{\ph_1+\ph_3,\ph_1+\ph_4\}$.
\end{lemma}

\begin{proof} By assumption, $\la_0=\frac{1}{2}$.

There exist distinct numbers $p,q\in\{1\sco r\}$ such that $\al:=\ep_p+\ep_q\in\De\sm H$.

Set $l:=1\in\R$ if $\la\in\ep_1+\{\ph_{r-1},\ph_r\}$ and $l:=r-1\in\R$ if $\la\notin\ep_1+\{\ph_{r-1},\ph_r\}$. Since $\la_0=\frac{1}{2}$ and
$\la\ne\ph_{r-1},\ph_r$, there exists a~subset $K\subs\Inn{\La}$ such that $(K-K)\cap\br{(\R\al)\sm\{0\}}=\es$ and
$\sums{x\in K}(x,\chk{\al})\ge l\cdot6\cdot2^{r-3}$. By Lemma~\ref{sxa}, $\bgm{\Inn{\La}\sm H}\ge l\cdot6\cdot2^{r-3}$,
\equ{
(\de l)\cdot6\cdot2^{r-3}\le\de\cdot\bgm{\Inn{\La}\sm H}\le|\De\sm H|+6\le|\De|+6=2r(r-1)+6,}
$(\de l)\cdot3\cdot2^{r-3}\le r(r-1)+3$, $1\le\de,l\le\de l\le\om$, where $\om:=\frac{r(r-1)+3}{3\cdot2^{r-3}}\in\R$. Hence, $r\le6$. Thus, $4\le r\le6$,
$l\le\om<3\le r-1$, implying $\la\in\ep_1+\{\ph_{r-1},\ph_r\}$. If $r\ne4$, then $r\in\{5,6\}$, $\de=2$ and $\om<2$, contradicting $\de\le\om$.

It follows from above that $r=4$ and $\la\in\{\ph_1+\ph_3,\ph_1+\ph_4\}$.
\end{proof}

\begin{lemma}\label{D_ev} Suppose that $\la\in Q\sm\De$ and $\la\ne2\ph_1$. If some hyperplane $H\subs\R^r$ satisfies
$\bgm{\Inn{\La}\sm H}\le|\De\sm H|+6$, then $r=4$ and $\la\in\{2\ph_3,2\ph_4\}$.
\end{lemma}

\begin{proof} By assumption, $\la_0=0$, $\la_1\sco\la_{r-1},|\la_r|\in\Z_{\ge0}$, and $\la_1\spl\la_{r-1}+|\la_r|\in4+2\Z_{\ge0}$. Therefore, using
p.~1---6 of the list above,
\begin{align*}
(r>4)\quad&\Ra\quad\br{\ep_1+\ep_2+\ep_3+\ep_4\in\Inn{\La}};\\
(r=4)\quad&\Ra\quad\Br{\br{2\ep_1+\ep_2+\ep_3\in\Inn{\La}}\lor(\la=\ep_1+\ep_2+\ep_3\pm\ep_4)},
\end{align*}
$\ph_2=\ep_1+\ep_2\in\Inn{\La}$, $\De\subs\Inn{\La}$, $\Inn{\La}\sm H=(\De\sm H)\sqcup\br{\Inn{\La}\sm(\De\cup H)}$,
$\bgm{\Inn{\La}\sm(\De\cup H)}=\bgm{\Inn{\La}\sm H}-|\De\sm H|\le6$.

There exists a~number $p\in\{1\sco r\}$ such that $\ep_p\notin H$. Further, let $\ga\in\Or(\R^r)$ be the reflection against the hyperplane
$(\R\ep_p)^{\perp}\subs\R^r$ and $\Ga\subs\Or(\R^r)$ the subgroup $\{E,\ga\}$, where $E$ denotes the identity operator. If $x\in\R^r$ and $x,\ga x\in H$,
then $x-\ga x\in\R\ep_p\cap H=0$, i.\,e. $x=\ga x$. Hence, if $K\subs\R^r$, $\Ga K=K$, and $K\cap(\R\ep_p)^{\perp}=\es$, then $K^{\Ga}=\es$,
$|K\sm H|\ge|K/\Ga|=\frac{1}{2}\cdot|K|$.

Assume that $r>4$. Then $\ep_1+\ep_2+\ep_3+\ep_4\in\Inn{\La}$. The subset $K\subs P$ of all vectors $\pm\ep_{i_1}\pm\ep_{i_2}\pm\ep_{i_3}\pm\ep_p\in P$,
where $i_1,i_2,i_3\in\{1\sco r\}\sm\{p\}$ are pairwise distinct numbers, satisfies $|K|=\frac{8}{3}\cdot(r-1)(r-2)(r-3)>16$, $K\subs\Inn{\La}\sm\De$,
$\Ga K=K$, and $K\cap(\R\ep_p)^{\perp}=\es$. Consequently, $\bgm{\Inn{\La}\sm(\De\cup H)}\ge|K\sm H|\ge\frac{1}{2}\cdot|K|>8$ contradicting
$\bgm{\Inn{\La}\sm(\De\cup H)}\le6$.

Suppose that $r=4$ and $\la\ne\ep_1+\ep_2+\ep_3\pm\ep_4$. Then $2\ep_1+\ep_2+\ep_3\in\Inn{\La}$. The subset $K\subs P$ of all vectors
$\pm\ep_i\pm\ep_j\pm2\ep_p\in P$ ($i,j\in\{1\sco r\}\sm\{p\}$, $i\ne j$) satisfies $|K|=4(r-1)(r-2)=24$, $K\subs\Inn{\La}\sm\De$, $\Ga K=K$, and
$K\cap(\R\ep_p)^{\perp}=\es$. Therefore, $\bgm{\Inn{\La}\sm(\De\cup H)}\ge|K\sm H|\ge\frac{1}{2}\cdot|K|=12$ contradicting
$\bgm{\Inn{\La}\sm(\De\cup H)}\le6$.

It follows from above that $r=4$ and $\la=\ep_1+\ep_2+\ep_3\pm\ep_4$.
\end{proof}

Let $\Pi'\subs\Pi\subs\R^r$ be an indecomposable simple system of order $r-2$.

The simple system $\Pi'\subs\Pi\subs\R^r$ corresponds to an indecomposable root system of rank $r-2$ and of order at least $(r-1)(r-2)$.
Consequently, $\bgm{\De\cap\ha{\Pi'}}\ge(r-1)(r-2)$, $\bgm{\De\sm\ha{\Pi'}}\le|\De|-(r-1)(r-2)=(r-1)(r+2)$.

\begin{lemma}\label{D_lah} Set $H:=\ba{\{\la\}\cup\Pi'}\subs\R^r$. If $r>4$, $(\la,\chk{\al}_{r-1})\eqi(\la,\chk{\al}_r)\pmod{2}$, and
$\la\notin Q\cup\{\ph_1\}$, then $\bgm{\Inn{\La}\sm H}>|\De\sm H|+6$.
\end{lemma}

\begin{proof} By assumption, $\la_0=0$, $\la_1\sco\la_{r-1},|\la_r|\in\Z_{\ge0}$, and $\la_1\spl\la_{r-1}+|\la_r|\in3+2\Z_{\ge0}$. Hence,
using p.~1---3, $\ep_1+\ep_2+\ep_3\in\Inn{\La}$. Note also that $|\De\sm H|\le\bgm{\De\sm\ha{\Pi'}}\le(r-1)(r+2)$.

There exist distinct numbers $p,q\in\{1\sco r\}$ such that $\al:=\ep_p+\ep_q\in\De\sm H$. For the subset $K\subs P$ of all vectors
$\ep_p\pm\ep_i\pm\ep_j\in P$, $\ep_q\pm\ep_i\pm\ep_j\in P$, and $\ep_p+\ep_q\pm\ep_i\in P$ ($i,j\in\{1\sco r\}\sm\{p,q\}$, $i\ne j$), we have
$K\subs\Inn{\La}$, $(K-K)\cap\br{(\R\al)\sm\{0\}}=\es$, and, besides,
\equ{
\sums{x\in K}(x,\chk{\al})=\sums{x\in K}(x_p+x_q)=(1+1)\cdot\Br{4\cdot\frac{(r-2)(r-3)}{2}}+2\cdot\br{2(r-2)}=4(r-2)^2.}
It follows from Lemma~\ref{sxa} that
$\bgm{\Inn{\La}\sm H}\bw\ge4(r-2)^2\bw=(3r-2)(r-5)+(r-1)(r+2)+8\bw\ge(r-1)(r+2)+8\bw\ge|\De\sm H|+8\bw>|\De\sm H|+6$.
\end{proof}

\begin{lemma}\label{D_lah1} Set $H:=\ba{\{\la\}\cup\Pi'}\subs\R^r$. If $\la=\ph_{r-1}$, $\Pi'=\{\al_3\sco\al_r\}\subs\Pi$, and
$\bgm{\Inn{\La}\sm H}\le|\De\sm H|+6$, then $r\le8$.
\end{lemma}

\begin{proof} By assumption, $|\La|=2^{r-1}$ and $H=\bc{x\in\R^r\cln x_1=x_2}\subs\R^r$ implying $|\La\sm H|\bw=\frac{1}{2}\cdot|\La|\bw=2^{r-2}$. Also,
$|\De\sm H|\le\bgm{\De\sm\ha{\Pi'}}\le(r-1)(r+2)$. Finally, $2^{r-2}\bw=|\La\sm H|\bw\le|\De\sm H|+6\bw\le(r-1)(r+2)+6\bw=r(r+1)+4$ and thus $r\le8$.
\end{proof}

\begin{prop}\label{D_all} Assume that $\la=\ph_1+\ph_{r-1}$. Set $\de:=1\bw\in\R$ if $r\in4\Z$ and $\de:=2\in\R$ if $r\notin4\Z$. Then
\begin{nums}{-1}
\item $2\la\notin\De\cup(\De+\De)$\~
\item there exist a~subset $\Om\subs\La$ and a~hyperplane $H\subs\R^r$ such that $\ha{\Om}=H$, $(\Om+\Om)\cap\De\bw=(\Om-\Om)\cap\De\bw=\es$, and
$\de\cdot\bgm{\Inn{\La}\sm H}>|\De\sm H|+6$.
\end{nums}
\end{prop}

\begin{proof} First of all note that $\la_1=\frac{3}{2}$, $2\la_1=3$, and, consequently, $2\la\notin\De\cup(\De+\De)$.

This proves the first statement. Let us pass to proving the second.

By assumption, $\la=\ep_1-\ep_r+\ph_r$. Therefore, $\ph_r-2\ep_r\in\La$, $\Om:=\{\ph_r-2\ep_j\cln j=2\sco r\}\subs\La$,
$H:=\ha{\Om}=\bc{x\in\R^r\cln(r-5)x_1=x_2\spl x_r}\subs\R^r$, and $|\Om|=\dim H=r-1$. Also, $(\Om+\Om)\cap\De=(\Om-\Om)\cap\De=\es$.

It remains to prove that $\de\cdot\bgm{\Inn{\La}\sm H}>|\De\sm H|+6$.

Suppose that $\de\cdot\bgm{\Inn{\La}\sm H}\le|\De\sm H|+6$. By Lemma~\ref{D_odd}, $r=4$, $|\La|=2^{r-1}\cdot r=32$, $|\De|=2r(r-1)=24$,
and $H=\bc{x\in\R^r\cln x_1+x_2+x_3+x_4=0}\subs\R^r$. It follows that $|\De\cap H|=r(r-1)=12$, $|\De\sm H|=12$, and, also,
$\La\cap H=\Om\sqcup(-\Om)\sqcup\bc{\pm(\ph_4-2\ep_1)}\subs\La$, $|\La\cap H|=2r=8$, $|\La\sm H|=24>|\De\sm H|+6$,
$\de\cdot\bgm{\Inn{\La}\sm H}\ge|\La\sm H|>|\De\sm H|+6$.

Thus, we come to a~contradiction, proving the claim.
\end{proof}

\begin{prop}\label{D_last} If $r=7$ and $\la=\ph_6$, then there exist a~subset $\Om\subs\La$ and a~hyperplane $H\subs\R^7$ such that $\ha{\Om}=H$,
$(\Om-\Om)\cap\De=\es$, and $2\cdot\bgm{\Inn{\La}\sm H}>|\De\sm H|+6$.
\end{prop}

\begin{proof} By assumption, $\la=\ph_7-\ep_7$, $|\La|=2^{r-1}=64$, and $|\De|=2r(r-1)=84$.

Set $\Om_0:=\bc{\ep_1+\ep_2+\ep_5,\ep_3+\ep_4+\ep_5,\ep_1+\ep_4+\ep_6,\ep_2+\ep_3+\ep_6,\ep_5+\ep_6+\ep_7}\subs\R^7$. The subset
$\Om:=\{-\ph_7\}\sqcup(\ph_7-\Om_0)\subs\La$ satisfies $(\Om-\Om)\cap\De=\es$ and $H:=\ha{\Om}\bw=\bc{x\in\R^7\cln x_1+x_3\bw=x_2+x_4}\bw\subs\R^7$.
Hence, $|\La\cap H|=24$, $|\La\sm H|=40$. We also have $\De\cap H\sups\De\cap\ha{\ep_5,\ep_6,\ep_7}$, $|\De\cap H|\ge12$, $|\De\sm H|\le72$,
$2\cdot|\La\sm H|=80>|\De\sm H|+6$.
\end{proof}

\section{Particular cases}\label{promain1}

In this section, we will prove Theorems \ref{B_main1}, \ref{C_main1}, and~\ref{D_main1}.

As before, we will use the conventions of \Ss\ref{introd}.

Set $n:=2r+1\in\N$ if $G\cong B_r$ and $n:=2r\in\N$ otherwise.

Suppose that $G=G^0$ and the linear algebra $R(\ggt_{\Cbb})$ is isomorphic to one of the following linear algebras\:
\begin{nums}{-1}
\item\label{adj} $\ad(\ggt_{\Cbb})$\~
\item\label{B_tav} $R_{\ph_1}(B_r)$\~
\item\label{B_sym} $R_{2\ph_1}(B_r)$\~
\item\label{B_symp} $R_{\ph_2}(B_2)\cong R_{\ph_1}(C_2)$\~
\item\label{C_sym} $R_{\ph_2}(C_r)$ \ter{$r>2$}\~
\item\label{C_tav} $R_{\ph_1}(C_r)$ \ter{$r>2$}\~
\item\label{C_symp} $R_{\ph_4}(C_4)$\~
\item\label{D_tav} $R_{\ph_1}(D_r)$ \ter{$r>3$}\~
\item\label{D_sym} $R_{2\ph_1}(D_r)$ \ter{$r>3$}\~
\item\label{D_fil} $R_{\ph_5}(D_5)$\~
\item\label{D_eil} $R_{\ph_8}(D_8)$.
\end{nums}

The following representations of complex simple Lie groups are polar (see~\cite[\Ss3]{CD})\:
\begin{itemize}
\item the adjoint representation of an arbitrary complex simple Lie group\~
\item the representations $R_{\ph_1}$ and~$R_{2\ph_1}$ of the complex simple Lie group $\SO_n(\Cbb)$\~
\item the representations $R_{\ph_2}$ and $2R_{\ph_1}$ of the complex simple Lie group $\Sp_{2r}(\Cbb)$\~
\item the representation~$R_{\ph_4}$ of the complex simple Lie group $\Sp_8(\Cbb)$\~
\item the representation $R_{\ph_5}+R'_{\ph_5}$ of the complex simple Lie group $\Spin_{10}(\Cbb)$\~
\item the representation~$R_{\ph_8}$ of the complex simple Lie group $\Spin_{16}(\Cbb)$.
\end{itemize}
Therefore, in each of the cases \ref{adj}---\ref{D_eil}, the linear Lie group $G\subs\GL(V)$ is polar and, by Lemma~\ref{pol}, the quotient $V/G$ is
homeomorphic to a~closed half-space (in particular, is not a~manifold).

Thus, we have completely proved Theorems \ref{B_main1}, \ref{C_main1}, and~\ref{D_main1}.

Theorems \ref{B_main}---\ref{D_main1} obviously imply Corollaries \ref{B_main2}---\ref{D_main2}.

\section*{Acknowledgements}

The author is very thankful to Prof. E.\,B.\?Vinberg for permanent support and valuable remarks and advice.

\end{document}